\documentclass[11pt]{amsart}
\usepackage[margin=1in]{geometry}
\usepackage{latexsym}
\usepackage{amsfonts}
\usepackage{amsmath}
\usepackage{amssymb}
\usepackage{amsthm}
\usepackage{enumerate}
\usepackage[hang,flushmargin]{footmisc}
\usepackage{caption}
\usepackage{mathrsfs}
\usepackage{amsaddr}
\usepackage{subfig}
\usepackage{mathtools}

\usepackage{xcolor}

\definecolor{MyBlue}{rgb}{0,0,1}
\definecolor{MyRed}{rgb}{1,0,0}
\definecolor{MyGreen}{rgb}{0,1,0}
\definecolor{MyIndigo}{rgb}{0.7254,0,1}
\definecolor{MyOrange}{rgb}{1,0.4431,0}

\usepackage{yhmath}
\usepackage{graphicx}
\usepackage{epstopdf}
\usepackage{epsfig}
\usepackage{caption}

\DeclareFontFamily{U}{mathx}{\hyphenchar\font45}
\DeclareFontShape{U}{mathx}{m}{n}{
      <5> <6> <7> <8> <9> <10>
      <10.95> <12> <14.4> <17.28> <20.74> <24.88>
      mathx10
      }{}
\DeclareSymbolFont{mathx}{U}{mathx}{m}{n}
\DeclareFontSubstitution{U}{mathx}{m}{n}
\DeclareMathAccent{\widecheck}{0}{mathx}{"71}
 
\usepackage{bm} 
\usepackage{cite}

\makeatletter

\newtheorem{theorem}{Theorem}[section]
\newtheorem{lemma}[theorem]{Lemma}
\newtheorem{proposition}[theorem]{Proposition}
\newtheorem{corollary}[theorem]{Corollary}

\newtheorem{conjecture}[theorem]{Conjecture}
\newtheorem{question}[theorem]{Question}

\theoremstyle{definition}
\newtheorem{definition}[theorem]{Definition}
\newtheorem{remarkx}[theorem]{Remark}
\newenvironment{remark}
  {\pushQED{\qed}\remarkx}
  {\popQED\endremarkx}
\newtheorem{examplex}[theorem]{Example}
\newenvironment{example}
  {\pushQED{\qed}\examplex}
  {\popQED\endexamplex}

\DeclareMathOperator{\conv}{conv}
\DeclareMathOperator{\NC}{NC}

\DeclareMathOperator{\VHC}{VHC}
\DeclareMathOperator{\id}{id}
\DeclareMathOperator{\des}{des}
\DeclareMathOperator{\peak}{peak}
\DeclareMathOperator{\black}{black}
\DeclareMathOperator{\skel}{skel}
\DeclareMathOperator{\tl}{tl}
\DeclareMathOperator{\Fer}{Fer}
\DeclareMathOperator{\Comp}{Comp}
\DeclareMathOperator{\QCan}{\mathsf{Qcan}}
\DeclareMathOperator{\Nest}{\mathsf{Nest}}

\makeatletter
\@namedef{subjclassname@2020}{%
  \textup{2020} Mathematics Subject Classification}
\makeatother

\newcommand{\dfn}[1]{\textcolor{blue}{\emph{#1}}}

\begin{document}
\title{Fertilitopes}
\author{Colin Defant}
\address{Princeton University \\ Department of Mathematics \\ Princeton, NJ 08544}
\email{cdefant@princeton.edu}

\begin{abstract}
We introduce tools from discrete convexity theory and polyhedral geometry into the theory of West's stack-sorting map $s$. Associated to each permutation $\pi$ is a particular set $\mathcal V(\pi)$ of integer compositions that appears in a formula for the fertility of $\pi$, which is defined to be $|s^{-1}(\pi)|$. These compositions also feature prominently in more general formulas involving families of colored binary plane trees called \emph{troupes} and in a formula that converts from free to classical cumulants in noncommutative probability theory. We show that $\mathcal V(\pi)$ is a transversal discrete polymatroid when it is nonempty. We define the \emph{fertilitope} of $\pi$ to be the convex hull of $\mathcal V(\pi)$, and we prove a surprisingly simple characterization of fertilitopes as nestohedra arising from full binary plane trees. Using known facts about nestohedra, we provide a procedure for describing the structure of the fertilitope of $\pi$ directly from $\pi$ using Bousquet-M\'elou's notion of the canonical tree of $\pi$. As a byproduct, we obtain a new combinatorial cumulant conversion formula in terms of generalizations of canonical trees that we call \emph{quasicanonical trees}. We also apply our results on fertilitopes to study combinatorial properties of the stack-sorting map. In particular, we show that the set of fertility numbers has density $1$, and we determine all infertility numbers of size at most $126$. Finally, we reformulate the conjecture that $\sum_{\sigma\in s^{-1}(\pi)}x^{\des(\sigma)+1}$ is always real-rooted in terms of nestohedra, and we propose natural ways in which this new version of the conjecture could be extended. 
\end{abstract}

\maketitle

\section{Introduction}\label{Sec:Intro} 

With the introduction of a certain \emph{stack-sorting machine} in his book \emph{The Art of Computer Programming} \cite{Knuth}, Knuth initiated the field of permutation patterns \cite{Bona, Kitaev, Linton} and also provided the first use of an invaluable tool in enumerative and analytic combinatorics called the \emph{kernel method}. West \cite{West} introduced a deterministic variant of Knuth's machine called the \dfn{stack-sorting map}. This map is a function, which we denote by $s$, that sends permutations of size $n$ to permutations of size $n$. In this article, a \dfn{permutation of size $n$} is an ordering of some set of $n$ positive integers (not necessarily the integers $1,\ldots,n$). For example, we consider $3179$ to be a permutation of size $4$. There is a simple recursive description of the stack-sorting map, which we now provide. First, $s$ sends the empty permutation to itself. Given a nonempty permutation $\pi$, we can write $\pi=LmR$, where $m$ is the largest entry in $\pi$. We then define $s(\pi)=s(L)s(R)m$. For example, \[s(4273615)=s(42)\,s(3615)\,7=s(2)\,4\,s(3)\,s(15)\,67=243\,s(1)\,567=2431567.\]

West's map has now become the most vigorously-studied form of stack-sorting; an exploration of this map and its close relatives winds through analytic combinatorics (see \cite{Bona, DefantTroupes, DefantCounting, DefantElvey} and the many references therein); combinatorial dynamics \cite{DefantPropp, DefantMonotonicity}; considerations of symmetric, unimodal, log-concave, and real-rooted polynomials \cite{BonaBoca, BonaSymmetry, Branden3, DefantCounting, DefantClass, DefantTroupes, DefantEngenMiller, Singhal}; special partially ordered sets \cite{DefantCatalan, DefantPolyurethane}; 
and even noncommutative probability theory \cite{DefantTroupes, DefantEngenMiller}. The goal of this paper is to introduce concepts from discrete convexity theory and polyhedral geometry into the theory of the stack-sorting map. 

\medskip
In order to provide a broad framework for generalizing many of the results known about stack-sorting, the author has introduced \emph{troupes}, which are essentially sets of colored binary plane trees that are closed under two operations called \emph{insertion} and \emph{decomposition}. \emph{Valid hook configurations} are combinatorial objects that are crucial for understanding the stack-sorting map and, more generally, postorder traversals of decreasing colored binary plane trees belonging to a troupe. Roughly speaking, a valid hook configuration of a permutation $\pi$ is a diagram obtained by decorating the plot of $\pi$ with rotated-L-shaped \emph{hooks} that satisfy special conditions (see Definition~\ref{Def5}). Quite surprisingly, valid hook configurations also appear naturally in a formula for converting from free to classical cumulants, allowing one to use tools from free probability theory to understand stack-sorting and vice-versa \cite{DefantTroupes}. Valid hook configurations have also been studied as combinatorial objects in their own right in \cite{DefantMotzkin, Maya, Ilani}.

If $\pi$ is a permutation of size $n$ with $k$ descents, then every valid hook configuration of $\pi$ induces a composition of the integer $n-k$ into $k+1$ parts; the compositions arising in this way are called the \dfn{valid compositions} of $\pi$, and the set of such compositions is denoted $\mathcal V(\pi)$. The \dfn{fertility} of $\pi$ is $|s^{-1}(\pi)|$, the number of preimages of $\pi$ under the stack-sorting map, and the Fertility Formula \cite{DefantPostorder, DefantTroupes} states that \[|s^{-1}(\pi)|=\sum_{(q_0,\ldots,q_k)\in\mathcal V(\pi)}\prod_{i=0}^kC_{q_i},\] where $C_r=\frac{1}{r+1}\binom{2r}{r}$ is the $r^\text{th}$ Catalan number. The Fertility Formula also has a much more general form, called the Refined Tree Fertility Formula, which applies in the setting of troupes and takes into account certain tree statistics (see Section~\ref{Subsec:FertilityFormulas} for more details). Let us also remark that a permutation $\pi$ satisfies $\mathcal V(\pi)=\{(1,1,\ldots,1)\}$ if and only if it is \emph{uniquely sorted}; such permutations are counted by a Lassalle's sequence and possess a great deal of unexpected enumerative structure \cite{DefantCatalan, DefantEngenMiller, Hanna, Singhal}. Research concerning valid compositions, uniquely sorted permutations, and fertility numbers also inspired the very recent article \cite{Cioni}, which deals with a different function called $\texttt{Queuesort}$.

In this paper, we view the valid compositions of $\pi$ as lattice points in $\mathbb R^{k+1}$, and we define the \dfn{fertilitope} of $\pi$, denoted $\Fer_\pi$, to be the convex hull of $\mathcal V(\pi)$. One of our main results states that $\mathcal V(\pi)$ is precisely the set of lattice points in $\Fer_{\pi}$ (see Theorem~\ref{Thm:Main1}). This provides a new perspective on the Fertility Formula and its generalizations, allowing us to express the fertility of $\pi$ as a weighted sum over the lattice points in a (convex) polytope. We will also see that we can convert from free to classical cumulants via a sum over the lattice points in a multiset of polytopes. 

To state the next main result of this article, we introduce some terminology from polyhedral combinatorics; we provide further details and examples in Section~\ref{Subsec:BinaryNestohedra}. Let $e_1,\ldots,e_{k+1}$ denote the standard basis vectors in $\mathbb R^{k+1}$. For each set $I\subseteq[k+1]\coloneqq\{1,\ldots,k+1\}$, let $\Delta_I$ denote the simplex obtained by taking the convex hull of the vectors $e_i$ for $i\in I$. Following \cite{Postnikov2}, we say a collection $\mathcal B$ of nonempty subsets of $[k+1]$ is a \dfn{building set} if it contains all singleton subsets of $[k+1]$ and has the property that if $I,J\in\mathcal B$ are not disjoint, then $I\cup J\in\mathcal B$. Given a building set $\mathcal B$ and a tuple ${\bf y}=(y_I)_{I\in\mathcal B}$ of positive real numbers indexed by the sets in $\mathcal B$, we can form the Minkowski sum \[\Nest({\bf y})=\sum_{I\in\mathcal B}y_I\Delta_I.\] A polytope $\Nest({\bf y})$ obtained in this manner is called a \dfn{nestohedron}. For more information about nestohedra, see \cite{Dao, Feichtner, Grujic, Pilaud, Postnikov2, Postnikov, Zelevinsky}.

We define a \dfn{binary building set} on $[k+1]$ to be a building set $\mathcal B$ such that for all $I,J\in \mathcal B$ with $I\cap J\neq\emptyset$, the sets $I$ and $J$ are intervals satisfying either $I\subseteq J$ or $J\subseteq I$ (the name comes from an equivalent definition in terms of binary plane trees that we give in Section~\ref{Subsec:BinaryNestohedra}). Given a binary building set $\mathcal B$, we can choose a tuple ${\bf y}=(y_I)_{I\in\mathcal B}\in\mathbb R_{>0}^{\mathcal B}$ and form the nestohedron $\Nest({\bf y})=\sum_{I\in\mathcal B}y_I\Delta_I$ as above. We define a \dfn{binary nestohedron} to be a nestohedron obtained in this fashion. A polytope in $\mathbb R^{k+1}$ is \dfn{integral} if its vertices lie in $\mathbb Z^{k+1}$. Our second main result (Theorem~\ref{Thm:Main2}) states that a nonempty polytope is a fertilitope if and only if it is an integral binary nestohedron. In particular, this implies that the set $\mathcal V(\pi)$ of valid compositions of a permutation $\pi$ is a transversal discrete polymatroid if it is nonempty (see Section~\ref{Subsec:BinaryNestohedra}). 

\medskip
In general, the set of valid compositions of a permutation $\pi$ can be quite complicated. Indeed, if valid compositions of arbitrary permutations were easy to describe, then the Fertility Formula would trivialize most inquiries regarding the stack-sorting map. Therefore, one who has thought a lot about the stack-sorting map and valid hook configurations would expect the set of all fertilitopes to be extremely unwieldy. This makes our simple characterization of fertilitopes as integral binary nestohedra all the more surprising.  

Our characterization of fertilitopes also allows us to glean information about their structure from known properties of nestohedra. For instance, we will see that the number of $i$-dimensional faces of the fertilitope of a permutation with $k$ descents is at most $\binom{k}{i}2^{k-i}$. When $\pi$ is a permutation with $\Fer_\pi\neq\emptyset$, it is natural to ask for an explicit description of a binary building set $\mathcal B$ and a tuple ${\bf y}=(y_I)_{I\in\mathcal B}$ such that $\Fer_\pi=\Nest({\bf y})$. We will provide such a description that makes use of the \emph{canonical tree} of $\pi$, which is a special decreasing binary plane tree with postorder traversal $\pi$ that was originally introduced by Bousquet-M\'elou \cite{Bousquet}.  Among other things, this allows us to read off the dimension and the face numbers of $\Fer_\pi$ from $\pi$. Along the way, we introduce a generalization of canonical trees that we call \emph{quasicanonical trees}. As an added bonus, we will obtain a new combinatorial formula for converting from free to classical cumulants in noncommutative probability theory that involves quasicanonical trees (see Corollary~\ref{Cor:Fertilitopes2}). This result adds to the recent surge of interest in combinatorial cumulant conversion formulas \cite{Arizmendi, Belinschi, Celestino, DefantTroupes, Ebrahimi, Lehner, Josuat}. 

\medskip
Our main results concerning fertilitopes (Theorems~\ref{Thm:Main1} and \ref{Thm:Main2}) will allow us to deduce new theorems about the stack-sorting map. One somewhat unexpected aspect of this map is that not every nonnegative integer arises as the fertility of a permutation. For example, there is no permutation with fertility $3$. In \cite{DefantFertility}, the author defined a \dfn{fertility number} to be a nonnegative integer that is the fertility of some permutation. A nonnegative integer that is not a fertility number (such as $3$) is called an \dfn{infertility number}. The main results from \cite{DefantFertility} are as follows: 
\begin{itemize}
\item The set of fertility numbers is closed under multiplication. 
\item If there exists a permutation of size $n$ with fertility $f$, then there exists a permutation of size $n+2$ with fertility $f$.
\item Every nonnegative integer that is not congruent to $3$ modulo $4$ is a fertility number. 
\item The smallest fertility number that is congruent to $3$ modulo $4$ is $27$. 
\item The lower asymptotic density of the set of fertility numbers is at least $0.7618$. 
\item If $f$ is a fertility number, then there is a permutation of size at most $f+1$ with fertility $f$.  
\end{itemize}

The set of fertility numbers forms a mysterious multiplicative monoid about which there are still several unanswered questions. In \cite{DefantFertility}, the author asked if the set of fertility numbers has a well-defined density. In this article, we answer this question in the affirmative by proving that it has density $1$; this greatly improves upon the lower bound of $0.7618$ from above. We also determine all infertility numbers that are at most $126$. When proving that certain numbers are infertility numbers, we will rely heavily on the fact that the set of valid compositions of a permutation is a discrete polymatroid (if it is nonempty). 

For each permutation $\pi$, we can consider the descent polynomial $\sum_{\sigma\in s^{-1}(\pi)}x^{\des(\sigma)+1}$, where $\des$ denotes the descent statistic. B\'ona \cite{BonaSymmetry} proved that these polynomials are symmetric and unimodal, and Br\"and\'en \cite{BrandenActions} strengthened B\'ona's theorem by proving that these polynomials are, in fact, $\gamma$-nonnegative. In \cite{DefantCounting}, the current author found a new proof of this $\gamma$-nonnegativity using valid hook configurations, and he stated the following conjecture. 

\begin{conjecture}[\!\!\cite{DefantCounting}]\label{Conj2} 
For each permutation $\pi$, the polynomial \[\sum_{\sigma\in s^{-1}(\pi)}x^{\des(\sigma)+1}\] has only real roots and is, therefore, log-concave. 
\end{conjecture}

Note that the log-concavity of the polynomials in Conjecture~\ref{Conj2} is also not known. 

\medskip

The following new conjecture requires us to consider the \dfn{Narayana polynomials} \[N_n(x)={\sum_{i=1}^n}\frac{1}{n}\binom{n}{i}\binom{n}{i-1}x^i.\] Given a composition ${\bf q}=(q_0,\ldots,q_k)\in\mathbb Z_{>0}^{k+1}$, we let $N_{\bf q}(x)=\prod_{i=0}^kN_{q_i}(x)$. 

\begin{conjecture}\label{Conj:Fertilitopes1}
For each integral binary nestohedron $P\subseteq\mathbb R^{k+1}$, the polynomial \[\sum_{{\bf q}\in P\cap\,\mathbb Z^{k+1}}N_{{\bf q}}(x)\] has only real roots and is, therefore, log-concave. 
\end{conjecture}

We will see that Conjectures~\ref{Conj2} and \ref{Conj:Fertilitopes1} are equivalent, despite the fact that they appear completely unrelated on the surface. In Section~\ref{Sec:Descent}, we will point to some natural directions in which one could extend Conjecture~\ref{Conj:Fertilitopes1}. Our hope is that such extensions could provide the correct framework within which to try proving Conjecture~\ref{Conj:Fertilitopes1} and, therefore, Conjecture~\ref{Conj2}. Because the set of lattice points in an integral binary nestohedron is a discrete polymatroid, one might hope that the structural properties of discrete polymatroids could help resolve Conjecture~\ref{Conj:Fertilitopes1}. Indeed, recent breakthroughs such as the theory of Lorentzian polynomials introduced by Br\"and\'en and Huh \cite{BrandenHuh} have found strong connections between discrete polymatroids (which are essentially the same as M-convex sets) and log-concave polynomials.

\subsection{Outline}

Section~\ref{Sec:Preliminaries} provides background information on valid hook configurations, valid compositions, troupes, tree traversals, cumulants, nestohedra, and discrete polymatroids. Section~\ref{Sec:Main} is devoted to proving our main theorems about valid compositions, fertilitopes, and binary nestohedra. In Section~\ref{Sec:Quasicanonical}, we introduce quasicanonical trees and use them to describe the structure of the fertilitope $\Fer_\pi$ directly from the plot of $\pi$; we also obtain a new cumulant conversion formula involving quasicanonical trees. Section~\ref{Sec:FertilityNumbers} proves that the set of fertility numbers has density $1$ and determines the complete list of infertility numbers that are at most $126$. Finally, Section~\ref{Sec:Descent} lists potential lines of inquiry about extensions of Conjecture~\ref{Conj:Fertilitopes1}; in the same section, we define \emph{extremal hook configurations}, which correspond to vertices of fertilitopes.  

\subsection{Terminology and Notation}

In this article, a \dfn{permutation} is an ordering of a finite set of positive integers, which we write in one-line notation. Let $S_n$ denote the set of permutations of the set $[n]\coloneqq\{1,\ldots,n\}$. Let $\id_n$ denote the identity permutation $123\cdots n$ in $S_n$. If $\pi$ is a permutation of size $n$, then the \dfn{standardization} of $\pi$ is the permutation in $S_n$ obtained by replacing the $i^\text{th}$-smallest entry in $\pi$ with $i$ for all $i$. For instance, the standardization of $3816$ is $2413\in S_4$. We say two permutations have the \dfn{same relative order} if their standardizations are equal. A \dfn{descent} of a permutation $\pi=\pi_1\cdots\pi_n$ is an index $i\in[n-1]$ such that $\pi_i>\pi_{i+1}$. A \dfn{peak} of $\pi$ is an index $i\in\{2,\ldots,n-1\}$ such that $\pi_{i-1}<\pi_i>\pi_{i+1}$. Let $\des(\pi)$ and $\peak(\pi)$ denote the number of descents of $\pi$ and the number of peaks of $\pi$, respectively. For $\pi\in S_n$ and $m\geq 0$, we let $\pi\oplus \id_m=\pi(n+1)\cdots(n+m)$ be the concatenation of $\pi$ with the increasing permutation $(n+1)\cdots(n+m)$ of size $m$. In particular, $\pi\oplus 1=\pi(n+1)$.

A \dfn{composition of $a$ into $b$ parts} is a tuple of $b$ positive integers, called the \dfn{parts} of the composition, that sum to $a$. If $(u_n)_{n\geq 1}$ is a sequence of elements of a field and ${\bf q}=(q_0,\ldots,q_k)$ is a composition, we let $u_{\bf q}=\prod_{i=0}^ku_{q_i}$. 

Every polytope in this article is assumed to be convex and to lie in a vector space $\mathbb R^n$ for some $n$. Let $\conv(A)$ denote the convex hull of a set $A\subseteq\mathbb R^n$. As mentioned above, we let $e_1,\ldots,e_n$ denote the standard basis vectors in $\mathbb R^n$ and let $\Delta_I=\conv(\{e_i:i\in I\})$. We will occasionally deal with Euclidean spaces of different dimensions; in these cases, it should be clear from context what the number of coordinates in the vector $e_i$ is supposed to be. The \dfn{Minkowski sum} of sets $A_1,\ldots,A_\ell\subseteq \mathbb R^n$ is the set $\sum_{i=1}^\ell A_i=\{a_1+\cdots+a_\ell:a_i\in A_i\text{ for all }i\in[\ell]\}$. The Minkowski sum of two sets $A_1$ and $A_2$ is written $A_1+A_2$. The \dfn{Cartesian product} of sets $A\subseteq\mathbb R^m$ and $B\subseteq\mathbb R^n$ is the set $A\times B=\{(x_1,\ldots,x_{m+n})\in\mathbb R^{m+n}:(x_1,\ldots,x_m)\in A, (x_{m+1},\ldots,x_{m+n})\in B\}$.

\section{Preliminaries}\label{Sec:Preliminaries}

\subsection{Valid Hook Configurations and Valid Compositions}\label{Subsec:VHCs}

The \dfn{plot} of a permutation $\pi=\pi_1\cdots\pi_n$ is the diagram showing the points $(i,\pi_i)$ for all $i\in[n]$. A \dfn{hook} of $\pi$ is a rotated L shape connecting two points $(i,\pi_i)$ and $(j,\pi_j)$ with $i<j$ and $\pi_i<\pi_j$, as in Figure~\ref{Fig4}. The point $(i,\pi_i)$ is the \dfn{southwest endpoint} of the hook, and $(j,\pi_j)$ is the \dfn{northeast endpoint} of the hook. For example, Figure~\ref{Fig4} shows the plot of the permutation $\pi=426315789$. The hook shown in this figure has southwest endpoint $(3,6)$ and northeast endpoint is $(8,8)$. If $i$ is a descent of $\pi$, then we call the point $(i,\pi_i)$ a \dfn{descent top} of the plot of $\pi$. 

\begin{figure}[ht]
  \begin{center}{\includegraphics[height=3.5cm]{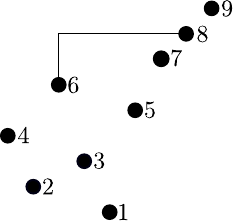}}
  \end{center}
  \caption{The plot of $426315789$ along with a single hook. We have labeled the points in the plot with their heights.}\label{Fig4}
\end{figure}

\begin{definition}\label{Def5}
Let $\pi$ be a permutation with descents $d_1<\cdots <d_k$. A \dfn{valid hook configuration} of $\pi$ is a tuple $\mathcal H=(H_1,\ldots,H_k)$ of hooks of $\pi$ that satisfy the following properties: 
\begin{enumerate}
\item \label{Item1} For each $i\in[k]$, the southwest endpoint of $H_i$ is $(d_i,\pi_{d_i})$. 
\item \label{Item2} No point in the plot of $\pi$ lies directly above a hook in $\mathcal H$. 
\item \label{Item3} No two hooks in $\mathcal H$ intersect or overlap each other unless the northeast endpoint of one is the southwest endpoint of the other. 
\end{enumerate}
Let $\VHC(\pi)$ denote the set of valid hook configurations of $\pi$. We make the convention that a valid hook configuration includes its underlying permutation as part of its identity so that $\VHC(\pi)\cap\VHC(\pi')=\emptyset$ when $\pi\neq \pi'$. Given a set $S$ of permutations, let $\VHC(S)=\bigcup_{\pi\in S}\VHC(\pi)$. We make the convention that if $\pi$ is monotonically increasing, then $\VHC(\pi)$ contains a single element: the empty valid hook configuration of $\pi$, which has no hooks. 
\end{definition}

Fix $\pi=\pi_1\cdots\pi_n$ with $\des(\pi)=k$. Each valid hook configuration $\mathcal H=(H_1,\ldots,H_k)\in\VHC(\pi)$ induces a coloring of the plot of $\pi$. To begin this coloring, draw a sky over the entire diagram and assign a color to the sky. Assign arbitrary distinct colors other than the color given to the sky to the hooks $H_1,\ldots,H_k$. There are $k$ northeast endpoints of hooks, and these points remain uncolored. However, all of the other $n-k$ points will be colored. In order to decide how to color a point $(i,\pi_i)$ that is not a northeast endpoint, imagine that this point looks directly upward. If it sees a hook when looking upward, it receives the same color as the hook that it sees. If it does not see a hook, it must see the sky, so it receives the same color as the sky. However, if $(i,\pi_i)$ is the southwest endpoint of a hook, then it must look around (on the left side of) the vertical part of that hook. Figure~\ref{Fig28} shows the coloring of the plot of a permutation induced by a valid hook configuration. 

\begin{figure}[ht]
\begin{center}
\includegraphics[height=7.525cm]{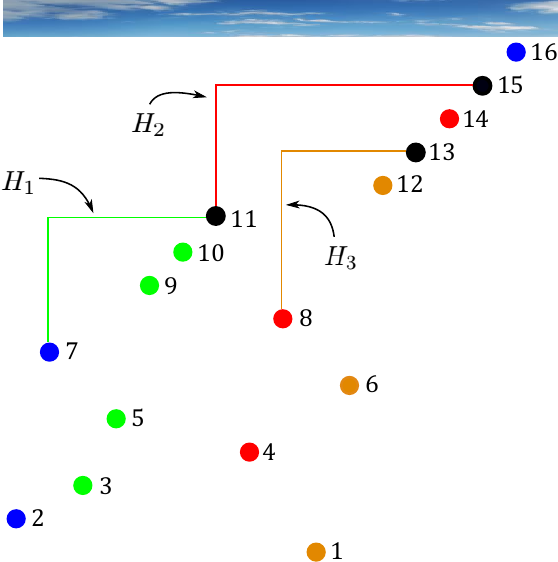}
\caption{A valid hook configuration and its induced coloring.} 
\label{Fig28}
\end{center}  
\end{figure} 

Let $q_i$ denote the number of points given the same color as the hook $H_i$. Let $q_0$ be the number of points given the same color as the sky. Let ${\bf q}^{\mathcal H}=(q_0,\ldots,q_k)$. For example, if $\mathcal H$ is the valid hook configuration shown in Figure~\ref{Fig28}, then ${\bf q}^{\mathcal H}=(3,4,3,3)$. Observe that the point $(d_i+1,\pi_{d_i+1})$ is necessarily the same color as $H_i$, while $(1,\pi_1)$ is necessarily the color of the sky. This means that ${\bf q}^{\mathcal H}$ is a composition of $n-k$ into $k+1$ parts. We say the valid hook configuration $\mathcal H$ \dfn{induces} this composition. A \dfn{valid composition} of $\pi$ is a composition induced by a valid hook 
configuration of $\pi$. Let \[\mathcal V(\pi)=\{q^{\mathcal H}:\mathcal H\in\VHC(\pi)\}\] denote the set of valid compositions of $\pi$. One simple property of valid compositions that we will use several times is that $\mathcal V(\pi)=\mathcal V(\pi')$ whenever $\pi$ and $\pi'$ have the same relative order. 

\begin{remark}\label{Rem6}
Fix a permutation $\pi$ with descents $d_1<\cdots<d_k$. The map $\VHC(\pi)\to\mathcal V(\pi)$ given by $\mathcal H\mapsto{\bf q}^{\mathcal H}$ is bijective. To see this, suppose $\mathcal H=(H_1,\ldots,H_k)\in\VHC(\pi)$ is such that ${\bf q}^{\mathcal H}=(q_0,\ldots,q_k)$. The hook $H_k$ is completely determined by the number $q_k$. Indeed, the southwest endpoint of $H_k$ must be $(d_k,\pi_{d_k})$, and the northeast endpoint of $H_k$ must be chosen so that exactly $q_k$ points lie underneath $H_k$. Once $H_k$ is drawn, the hook $H_{k-1}$ is similarly determined by the number $q_{k-1}$. Continuing in this fashion, we find that all of the hooks $H_k,\ldots,H_1$ are uniquely determined.  
\end{remark} 

The next subsections detail why valid compositions are so important in the theory of stack-sorting and, more generally, in the theory of troupes. We will also briefly discuss how they relate to cumulants in noncommutative probability theory.

\subsection{Troupes}

A \dfn{binary plane tree} is a rooted tree in which each vertex has at most $2$ children and each child is designated as either a left or a right child. Let us fix a finite set ${\bf C}$ of colors; assume that ${\bf C}$ contains the colors black and white. A \dfn{colored binary plane tree} is a tree obtained from a binary plane tree by assigning each vertex a color from ${\bf C}$. Let $\mathsf{BPT}$ and $\mathsf{CBPT}$ denote the set of binary plane trees and the set of colored binary plane trees, respectively. We make the convention that a binary plane tree is a colored binary plane tree in which all vertices are black; thus, $\mathsf{BPT}\subseteq\mathsf{CBPT}$. For ${\bf T}\subseteq\mathsf{CBPT}$, we let ${\bf T}_n$ denote the set of trees in ${\bf T}$ that have $n$ vertices.

Let $T_1$ and $T_2$ be nonempty colored binary plane trees, and let $v$ be a vertex of $T_1$. Let us replace $v$ with two vertices that are connected by a left edge. This produces a new tree $T_1^*$ with one more vertex than $T_1$. We call the lower endpoint of the new left edge $v$, identifying it with the original vertex $v$ and giving it the same color as the original $v$. We denote the upper endpoint of the new left edge by $v^*$, and we color $v^*$ black. For instance, if \[T_1=\begin{array}{l}\includegraphics[height=1.306cm]{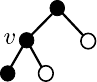}\end{array},\] where $v$ is as indicated, then \[T_1^*=\begin{array}{l}\includegraphics[height=1.835cm]{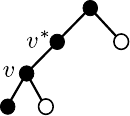}\end{array}.\] The \dfn{insertion of $T_2$ into $T_1$ at $v$}, denoted $\nabla_v(T_1,T_2)$, is the tree formed by attaching $T_2$ as the right subtree of $v^*$ in $T_1^*$. For example, if $T_1$ and $T_1^*$ are as above and \[T_2=\begin{array}{l}\includegraphics[height=1.309cm]{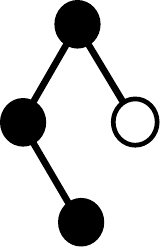}\end{array},\] then \[\nabla_v(T_1,T_2)=\begin{array}{l}\includegraphics[height=2.358cm]{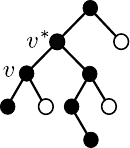}\end{array}.\] 

One can reverse the above procedure. Let $T$ be a colored binary plane tree, and suppose $v^*$ is a black vertex in $T$ with $2$ children. Let $v$ be the left child of $v^*$ in $T$, and let $T_2$ be the right subtree of $v^*$ in $T$. Let $T_1^*$ be the tree obtained by deleting $T_2$ from $T$, and let $T_1$ be the tree obtained from $T_1^*$ by contracting the edge connecting $v$ and $v^*$ into a single vertex. We call this contracted vertex $v$, identifying it with (and giving it the same color as) the original $v$. We say the pair $(T_1,T_2)$ is the \dfn{decomposition of $T$ at $v^*$} and write $\Delta_{v^*}(T)=(T_1,T_2)$.  

We say a collection ${\bf T}$ of colored binary plane trees is 
\begin{itemize}
\item \dfn{insertion-closed} if for all nonempty trees $T_1,T_2\in{\bf T}$ and every vertex $v$ of $T_1$, the tree $\nabla_v(T_1,T_2)$ is in ${\bf T}$; 
\item \dfn{decomposition-closed} if for every $T\in{\bf T}$ and every black vertex $v^*$ of $T$ that has $2$ children, the pair $\Delta_{v^*}(T)$ is in ${\bf T}\times{\bf T}$; 
\item \dfn{black-peaked} if for every $T\in{\bf T}$, the vertices with $2$ children in $T$ are all black.
\end{itemize} A \dfn{troupe} is a set of colored binary plane trees that is insertion-closed, decomposition-closed, and black-peaked. The article \cite{DefantTroupes} shows that there are uncountably many troupes (even if we only allow black vertices), gives several important examples of troupes, and characterizes troupes in terms of their \emph{branch generators}, which are essentially their ``indecomposable'' elements. The only specific troupe that we will need in this article is $\mathsf{BPT}$, the troupe of binary plane trees.

\subsection{Tree Traversals}\label{Subsec:Traversals}

If $X$ is a finite set of positive integers, then a \dfn{decreasing colored binary plane tree on $X$} is a colored binary plane tree whose vertices are bijectively labeled with the elements of $X$ so that every nonroot vertex has a label that is smaller than the label of its parent. If $\mathcal T$ is a decreasing colored binary plane tree, then the \dfn{skeleton} of $\mathcal T$, denoted $\skel(\mathcal T)$, is the colored binary plane tree obtained by removing the labels from $\mathcal T$. Given a set ${\bf T}$ of colored binary plane trees, we let $\mathsf{D}{\bf T}$ denote the set of decreasing colored binary plane trees $\mathcal T$ such that $\skel(\mathcal T)\in{\bf T}$. Thus, $\mathsf{DCBPT}$ is the set of all decreasing colored binary plane trees. 

The \dfn{in-order traversal} $\mathcal I$ and the \dfn{postorder traversal} $\mathcal P$ are two functions that send decreasing colored binary plane trees to permutations. If $\mathcal T$ is the empty tree, then $\mathcal I(\mathcal T)$ and $\mathcal P(\mathcal T)$ are both just the empty permutation. Now suppose $\mathcal T$ is a nonempty decreasing colored binary plane tree on $X$. Let $\mathcal T_L$ and $\mathcal T_R$ be the (possibly empty) left and right subtrees of the root of $\mathcal T$, respectively. Let $m=\max(X)$ be the label of the root of $\mathcal T$. The in-order and postorder traversals are defined recursively by \[\mathcal I(\mathcal T)=\mathcal I(\mathcal T_L)\,m\,\mathcal I(\mathcal T_R)\quad\text{and}\quad\mathcal P(\mathcal T)=\mathcal P(\mathcal T_L)\,\mathcal P(\mathcal T_R)\,m.\] It is known that $\mathcal I$ is a bijection from the set of decreasing binary plane trees on $X$ (recall that this means all vertices are black) to the set of permutations of $X$. Therefore, given a permutation $\pi$, we let $\mathcal I^{-1}(\pi)$ denote the unique binary plane tree whose in-order traversal is $\pi$. The connection between tree traversals and stack-sorting comes from the identity
\begin{equation}\label{sPcircI}
s=\mathcal P\circ\mathcal I^{-1}
\end{equation} (see, e.g., \cite[Chapter 8]{Bona}). For example, \[246153\xrightarrow{\mathcal I^{-1}}\begin{array}{l}\includegraphics[height=1.339cm]{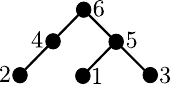}\end{array}\xrightarrow{\,\,\,\mathcal P\,\,\,}241356,\] and $s(246153)=s(24)\,s(153)\,6=s(2)\,4\,s(1)\,s(3)\,56=241356$.

If $T$ is a colored binary plane tree or a decreasing colored binary plane tree, then we let $\des(T)$, $\peak(T)$, and $\black(T)$ denote, respectively, the number of right edges of $T$, the number of vertices of $T$ that have $2$ children, and the number of black vertices of $T$. The use of the notation $\des$ comes from the easily-verified fact that the number of right edges of a decreasing binary plane tree $\mathcal T$ is the number of descents of $\mathcal I(\mathcal T)$. Similarly, the number of vertices of $\mathcal T$ that have $2$ children is the number of peaks of $\mathcal I(\mathcal T)$. 

In this article, a \dfn{tree statistic} is a function $f:\mathsf{CBPT}\to\mathbb C$. If we are given a tree statistic $f:\mathsf{CBPT}\to\mathbb C$, we define a function $\ddot f:\mathsf {DCBPT}\to\mathbb C$ by \[\ddot f(\mathcal T)=f(\skel(\mathcal T)).\] We say a tree statistic $f$ is \dfn{insertion-additive} if \[f(\nabla_v(T_1,T_2))=f(T_1)+f(T_2)\] for all nonempty colored binary plane trees $T_1$ and $T_2$ and all vertices $v$ in $T_1$. The maps $T\mapsto\des(T)+1$, $T\mapsto\peak(T)+1$, and $T\mapsto\black(T)+1$ are all examples of insertion-additive tree statistics. 

\subsection{Fertility Formulas}\label{Subsec:FertilityFormulas}

Recall that the fertility of a permutation $\pi$ is defined to be $|s^{-1}(\pi)|$. In his thesis, West \cite{West} detailed many lengthy computations in order to find the fertilities of some very specific types of permutations. Bousquet-M\'elou found an algorithm for determining whether or not a permutation has fertility $0$, and she asked for a general method for computing the fertility of an arbitrary permutation. The author has answered this question in several forms. One answer is given by the Fertility Formula, which we state below. First, we state the much more general Refined Tree Fertility Formula, from which the Fertility Formula easily follows. This result illustrates the utility of valid compositions. The article \cite{DefantTroupes} details several applications of this result to the study of stack-sorting, generalizations of stack-sorting, and the relationship between free and classical cumulants.

\begin{theorem}[Refined Tree Fertility Formula \cite{DefantTroupes}]\label{RTFF} 
Let ${\bf T}$ be a troupe, and let $f_1,\ldots,f_r$ be insertion-additive tree statistics. For every permutation $\pi$, we have \[\sum_{\mathcal T\in\mathcal P^{-1}(\pi)\cap\mathsf D{\bf T}}x_1^{\ddot f_1(\mathcal T)}\cdots x_r^{\ddot f_r(\mathcal T)}=\sum_{(q_0,\ldots,q_k)\in\mathcal V(\pi)}\prod_{t=0}^k\sum_{T\in{\bf T}_{q_t}}x_1^{f_1(T)}\cdots x_r^{f_r(T)}.\] 
\end{theorem}

In order to make the previous result more palatable, let us specialize it to the specific cases that will be most useful to us later. First, define the \dfn{Narayana numbers} $N(n,i)=\frac{1}{n}\binom{n}{i}\binom{n}{i-1}$. These numbers constitute one of the most common refinements of the sequence of Catalan numbers. We define the $n^\text{th}$ \dfn{Narayana polynomial} by $N_n(x)=\sum_{i=1}^nN(n,i)x^i$. It is known \cite{BonaSymmetry, DefantTroupes} that 
\begin{equation}\label{EqNar}
N_n(x)=\sum_{T\in\mathsf{BPT}_n}x^{\des(T)+1}=\sum_{\mathcal T\in\mathsf{DBPT}_n\cap\mathcal P^{-1}(\id_n)}x^{\des(\mathcal T)+1}=\sum_{\sigma\in s^{-1}(\id_n)}x^{\des(\sigma)+1}.
\end{equation} Let us remark that the identity $\displaystyle\sum_{T\in\mathsf{BPT}_n}x^{\des(T)+1}=\sum_{\mathcal T\in\mathsf{DBPT}_n\cap\mathcal P^{-1}(\id_n)}x^{\des(\mathcal T)+1}$ follows from the well-known (and easily-verified) fact that every binary plane tree with $n$ vertices can be labeled in a unique way to obtain a decreasing binary plane tree whose postorder traversal is $\id_n$. The identity $\displaystyle \sum_{\mathcal T\in\mathsf{DBPT}_n\cap\mathcal P^{-1}(\id_n)}x^{\des(\mathcal T)+1}=\sum_{\sigma\in s^{-1}(\id_n)}x^{\des(\sigma)+1}$ follows from \eqref{sPcircI} and the fact that $\des(\sigma)=\des(\mathcal I^{-1}(\sigma))$ for every permutation $\sigma$. For a composition ${\bf q}=(q_0,\ldots,q_k)$, let $N_{\bf q}(x)=\prod_{i=0}^kN_{q_i}(x)$.

The next corollary is a special case of the Refined Fertility Formula from \cite{DefantTroupes}. It was essentially first proven in \cite{DefantPostorder}, and it has been used in \cite{DefantCounting, DefantClass, DefantTroupes} to analyze the descent polynomials of stack-sorting preimages of various sets of permutations. For example, this result easily implies B\'ona's main result from \cite{BonaSymmetry}, which states that $\sum_{\sigma\in s^{-1}(\pi)}x^{\des(\sigma)+1}$ is a symmetric and unimodal polynomial for every permutation $\pi$. To prove the following corollary, we simply combine \eqref{EqNar} with the special case of Theorem~\ref{RTFF} in which ${\bf T}=\mathsf{BPT}$, $r=1$, $x_1=x$, and $f(T)=\des(T)+1$. 

\begin{corollary}[Refined Fertility Formula \cite{DefantPostorder, DefantTroupes}]\label{RFF'}
For every permutation $\pi$, we have \[\sum_{\sigma\in s^{-1}(\pi)}x^{\des(\sigma)+1}=\sum_{{\bf q}\in\mathcal V(\pi)}N_{{\bf q}}(x).\] 
\end{corollary} 

Finally, by setting $x=1$ in the previous corollary, we obtain the Fertility Formula, which has been a crucial tool for analyzing the stack-sorting map in recent years \cite{DefantCatalan, DefantClass, DefantFertility, DefantFertilityWilf, DefantEngenMiller, DefantPreimages}. 

\begin{corollary}[Fertility Formula \cite{DefantPostorder, DefantTroupes}]\label{FF}
For every permutation $\pi$, we have \[|s^{-1}(\pi)|=\sum_{{\bf q}\in\mathcal V(\pi)}C_{\bf q}.\] 
\end{corollary}

Following Bousquet-M\'elou, we say a permutation $\pi$ is \dfn{sorted} if its fertility is positive (i.e., it is in the image of $s$). It follows from the Fertility Formula that $\pi$ is sorted if and only if $\mathcal V(\pi)\neq\emptyset$. 

\begin{remark}\label{Rem1} 
Suppose $\pi$ is a permutation that has a valid hook configuration. Then $\pi$ is sorted, so it is immediate from the definition of $s$ that $\pi$ ends in its largest entry.
\end{remark}

\subsection{Cumulants}
Free probability is a relatively new area of mathematics that began with the work of Voiculescu \cite{Voiculescu, Voiculescu2} in the 1980's; it has found applications in a wide variety of mathematical fields, now including stack-sorting. In this subsection, we discuss the VHC Cumulant Formula, which is responsible for the surprising and useful connection between stack-sorting and free probability. This formula first emerged in \cite{DefantEngenMiller} and was formalized in \cite{DefantTroupes}. We will hardly discuss any details of free probability theory; a standard reference that the interested reader can consult is \cite{Nica}.

Let $\Pi(n)$ denote the collection of all set partitions of the set $[n]$. Given $\rho\in\Pi(n)$, we say two distinct blocks $B,B'\in\rho$ form a \dfn{crossing} if there exist $i,j\in B$ and $i',j'\in B'$ such that either $i<i'<j<j'$ or $i>i'>j>j'$. A partition is \dfn{noncrossing} if no two of its blocks form a crossing. Let $\NC(n)$ denote the collection of all noncrossing partitions of $[n]$. 

Let $\mathbb K$ be a field, and let $(m_n)_{n\geq 1}$ be a sequence of elements of $\mathbb K$, which we call a \dfn{moment sequence}. Associated to this moment sequence is the sequence $(c_n)_{n\geq 1}$ of \dfn{classical cumulants}, which is defined implicitly by the identity \[m_n=\sum_{\rho\in\Pi(n)}\prod_{B\in \rho}c_{|B|}.\] Similarly, there is a sequence $(\kappa_n)_{n\geq 1}$ of \dfn{free cumulants} associated to the moment sequence, which is defined implicitly by the identity \[m_n=\sum_{\eta\in\NC(n)}\prod_{B\in \rho}\kappa_{|B|}.\] Speicher introduced free cumulants in \cite{Speicher} in order to afford a combinatorial approach to Voiculescu's free probability theory. 

Any one of the sequences $(m_n)_{n\geq 1}$, $(c_n)_{n\geq 1}$, $(\kappa_n)_{n\geq 1}$ determines the other two. In particular, if we are given a sequence $(\kappa_n)_{n\geq 1}$ of free cumulants, then there is a unique corresponding sequence $(c_n)_{n\geq 1}$ of classical cumulants. Two other types of cumulants that will not concern us, but which have been studied in the literature, are \emph{Boolean cumulants} and \emph{monotone cumulants}. There has been a lot of attention in recent years devoted to finding formulas that convert from one type of cumulant sequence to another \cite{Arizmendi, Belinschi, Celestino, DefantTroupes, Ebrahimi, Lehner, Josuat}. The VHC Cumulant Formula, which we now state, is one such formula. What makes this formula different from the others is that it has applications to the combinatorics of troupes and stack-sorting, topics that one would not expect to have any connections with cumulants. Many of these applications are described in \cite{DefantTroupes}. Given a sequence $(\kappa_n)_{n\geq 1}$ and a composition ${\bf q}=(q_0,\ldots,q_k)$, we write $(-\kappa_{\bullet+1})_{\bf q}$ for the product $\prod_{i=0}^p(-\kappa_{q_i+1})$. 

\begin{theorem}[VHC Cumulant Formula \cite{DefantTroupes}]\label{VHCCF}
If $(\kappa_n)_{n\geq 1}$ is a sequence of free cumulants, then the corresponding classical cumulants are given by \[-c_n=\sum_{\pi\in S_{n-1}}\sum_{{\bf q}\in\mathcal V(\pi)}(-\kappa_{\bullet+1})_{\bf q}.\] 
\end{theorem}

\begin{remark}
As discussed in \cite{DefantTroupes}, the VHC Cumulant Formula actually extends to the more general setting of \emph{multivariate cumulants}. 
\end{remark}

For an explicit example of the VHC Cumulant Formula, suppose $(\kappa_n)_{n\geq 1}$ is a sequence of free cumulants, and let $(c_n)_{n\geq 1}$ be the corresponding sequence of classical cumulants. Let us express $c_4$ in terms of the free cumulants $\kappa_n$. The only valid hook configurations of permutation in $S_3$ are (drawn with their induced colorings)
\[\begin{array}{l}\includegraphics[height=1.369cm]{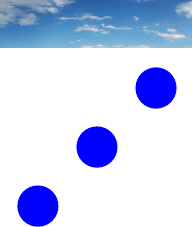}\end{array}\quad\text{and}\quad\begin{array}{l}\includegraphics[height=1.369cm]{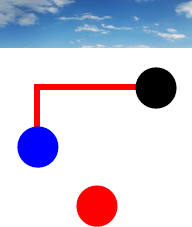}\end{array}.\] The valid compositions induced by these valid hook configurations are $(3)$ and $(1,1)$. Therefore, Theorem~\ref{VHCCF} tells us that $-c_4=(-\kappa_{3+1})+(-\kappa_{1+1})(-\kappa_{1+1})=-\kappa_4+\kappa_2^2$. 

\subsection{Binary Nestohedra}\label{Subsec:BinaryNestohedra}
We defined binary nestohedra in Section~\ref{Sec:Intro}; we now rephrase this definition in terms of trees and state some properties of these polytopes. As before, we let $e_1,\ldots,e_{k+1}$ denote the standard basis vectors in $\mathbb R^{k+1}$. Let $\conv(Q)$ denote the convex hull of a set $Q\subseteq\mathbb R^{k+1}$. For $I\subseteq[k+1]$, let $\Delta_I=\conv(\{e_i:i\in I\})$.  

A \dfn{building set on $[k+1]$} is a collection $\mathcal B$ of nonempty subsets of $[k+1]$ that contains all singletons $\{1\},\ldots,\{k+1\}$ and has the property that if $I,J\in\mathcal B$ and $I\cap J\neq\emptyset$, then $I\cup J\in\mathcal B$. Given a building set $\mathcal B$ and a tuple ${\bf y}=(y_I)_{I\in\mathcal B}\in\mathbb R_{>0}^{\mathcal B}$ of positive real numbers indexed by the sets in $\mathcal B$, we can form the polytope $\Nest({\bf y})=\sum_{I\in\mathcal B}({\bf y})$, where the sum denotes Minkowski sum. A polytope $\Nest({\bf y})$ obtained in this manner is called a \dfn{nestohedron}.

A binary plane tree is called \dfn{full} if each of its vertices has either $0$ or $2$ children. Choose a full binary plane tree $T$ with $k+1$ leaves (and $2k+1$ vertices in total), and identify the leaves with the singleton sets $\{1\},\ldots,\{k+1\}$ from left to right. Identify each internal vertex of $T$ with the union of the leaves lying weakly below it in $T$. Notice that each vertex is an interval in $[k+1]$, meaning that it is of the form $\{i,\ldots,j\}$ for some $1\leq i\leq j\leq k+1$. Now choose a set $\mathcal B$ of vertices of $T$ such that every singleton set (i.e., every leaf) is in $\mathcal B$. We call a collection of sets $\mathcal B$ obtained in this way a \dfn{binary building set}; note that a binary building set is indeed a building set. Equivalently, one can define a binary building set on $[k+1]$ to be a collection of nonempty subsets of $[k+1]$ such that 
\begin{itemize}
\item every set in $\mathcal B$ is an interval;
\item every singleton subset of $[k+1]$ is in $\mathcal B$;
\item if $I,J\in\mathcal B$ are not disjoint, then either $I\subseteq J$ or $J\subseteq I$. 
\end{itemize}

Figure~\ref{Fig1} shows a full binary plane tree with $5$ leaves, where we have identified the vertices of the tree with subsets of $\{1,2,3,4,5\}$. Let $\mathcal B=\{\{1\},\{2\},\{3\},\{4\},\{5\},\{2,3\},\{1,2,3,4,5\}\}$ be the collection of vertices that are drawn in boxes. 

\begin{figure}[ht]
  \begin{center}{\includegraphics[height=3.448cm]{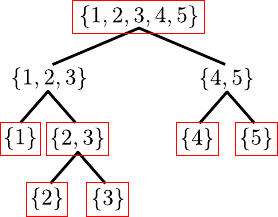}}
  \end{center}
  \caption{A full binary plane tree with vertices identified with subsets of $\{1,2,3,4,5\}$. The sets in boxes form a binary building set.}\label{Fig1}
\end{figure}

Given a binary building set $\mathcal B$ on $[k+1]$, we can choose a tuple ${\bf y}=(y_I)_{I\in\mathcal B}\in\mathbb R_{>0}^{\mathcal B}$ of positive real numbers indexed by the sets in $\mathcal B$ and form the nestohedron $\Nest({\bf y})=\sum_{I\in\mathcal B}y_I\Delta_I$ as above. We call a nestohedron obtained from a binary building set in this manner a \dfn{binary nestohedron}. We make the convention that binary nestohedra are nonempty. 

We now review some facts about nestohedra from \cite{Feichtner, Postnikov2, Postnikov}. To begin, let us fix a building set $\mathcal B$ on $[k+1]$. Let $\mathcal B_{\max}$ be the collection of sets in $\mathcal B$ that are maximal by inclusion. Following \cite{Postnikov, Zelevinsky}, we say a set $N\subseteq\mathcal B\setminus\mathcal B_{\max}$ is a \dfn{nested set} for $\mathcal B$ if 
\begin{itemize}
\item for all $I,J\in N$ with $I\cap J\neq\emptyset$, we have either $I\subseteq J$ or $J\subseteq I$;
\item for all $\ell\geq 2$ and all pairwise-disjoint sets $J_1,\ldots,J_\ell\in N$, the union $J_1\cup\cdots\cup J_\ell$ is not in $\mathcal B$.
\end{itemize} Observe that the first bulleted condition is automatically satisfied if $\mathcal B$ is a binary building set. The \dfn{nested set complex} $\Delta_{\mathcal B}$ is the simplicial complex consisting of all nested sets for $\mathcal B$. The next theorem was discovered independently in \cite{Feichtner} and \cite{Postnikov2}. 

\begin{theorem}[\!\!\cite{Feichtner, Postnikov2}]\label{Thm:Feichtner}
Let $\mathcal B$ be a building set on $[k+1]$. For any tuple ${\bf y}=(y_I)_{I\in\mathcal B}\in\mathbb R_{>0}^{\mathcal B}$, the nestohedron $\Nest({\bf y})$ is a simple polytope of dimension $k+1-|\mathcal B_{\max}|$. The dual simplicial complex of $\Nest({\bf y})$ is isomorphic to the nested set complex $\Delta_{\mathcal B}$. 
\end{theorem} 

The previous theorem tells us that there is a bijection between nonempty faces of $\Nest({\bf y})$ and nested sets for $\mathcal B$. We can make this bijection explicit as follows. First, for each set $I\in\mathcal B$, let $z_I=\sum_{J\subseteq I}y_J$, where the sum is over sets $J\in\mathcal B$ that are contained in $I$. Given a nested set $N$ for $\mathcal B$, we define 
\begin{equation}\label{Eq:FN}
F_N=\left\{(x_1,\ldots,x_{k+1})\in\mathbb R^{k+1}:\sum_{i\in I}x_i=z_I\text{ for all }I\in N\text{ and }\sum_{j\in J}x_j\geq z_J\text{ for all }J\in\mathcal B\right\}.
\end{equation}

\begin{theorem}[\!\!\cite{Postnikov2}]\label{Thm:PostnikovFN}
Let $\mathcal B$ be a building set on $[k+1]$, and let ${\bf y}\in\mathbb R_{>0}^{\mathcal B}$. The map $N\mapsto F_N$ gives a bijection between nested sets for $\mathcal B$ and nonempty faces of $\Nest({\bf y})$. Moreover, we have $\dim(F_N)=k+1-|\mathcal B_{\max}|-|N|$ for every nested set $N$. 
\end{theorem}

Recall that a polytope $P\subseteq\mathbb R^{k+1}$ is \dfn{integral} if all of its vertices are in $\mathbb Z^{k+1}$. The following corollary will allow us to phrase one of our main results in the next section more concisely. 

\begin{corollary}\label{Cor1}
Let $\mathcal B$ be a building set on $[k+1]$, and let ${\bf y}=(y_I)_{I\in\mathcal B}\in\mathbb R_{>0}^{\mathcal B}$. The nestohedron $\Nest({\bf y})$ is integral if and only if $y_I\in\mathbb Z$ for all $I\in\mathcal B$.  
\end{corollary}

\begin{proof}
If $y_I\in\mathbb Z$ for all $I\in\mathcal B$, then $\Nest({\bf y})$ is integral because it is a Minkowski sum of integral polytopes. To prove the converse, suppose there exists $I\in\mathcal B$ with $y_I\not\in\mathbb Z$. We may assume $y_J\in\mathbb Z$ for all sets $J\in\mathcal B$ that are properly contained in $I$. Then $z_I\not\in\mathbb Z$. The set $N=\{I\}$ is clearly a nested set for $\mathcal B$, so, by Theorem~\ref{Thm:PostnikovFN}, it corresponds to a face $F_N$ of ${\bf P}({\bf y})$. Let $(v_1,\ldots,v_{k+1})$ be a vertex in $F_N$. It follows from \eqref{Eq:FN} that $\sum_{i\in I}v_i=z_I\not\in\mathbb Z$, so $\Nest({\bf y})$ is not integral.
\end{proof}

\begin{remark}
Let us define $\mathsf{IBN}$ to be the smallest (under containment) set of polytopes satisfying the following properties: 
\begin{itemize}
\item The $0$-dimensional polytope $\{(1)\}\subseteq \mathbb R$ is in $\mathsf{IBN}$. 
\item If $P\subseteq\mathbb R^{k+1}$ is a polytope in $\mathsf{IBN}$, then $P+\Delta_{[k+1]}$ is in $\mathsf{IBN}$. 
\item If $P,Q\in\mathsf{IBN}$, then $P\times Q\in\mathsf{IBN}$.
\end{itemize}
It is straightforward to check that $\mathsf{IBN}$ is precisely the set of integral binary nestohedra. 
\end{remark} 

We will also need the following result due to Postnikov concerning lattice points in Minkowski sums of coordinate simplices. In particular, this result describes the lattice points in an integral nestohedron. 

\begin{theorem}[\!\!{\cite[Proposition 14.12]{Postnikov2}}]\label{Thm:PostLattice}
If $I_1,\ldots,I_\ell$ are (not necessarily distinct) subsets of $[k+1]$, then \[\left(\sum_{j=1}^\ell\Delta_{I_j}\right)\cap\,\mathbb Z^{k+1}=\sum_{j=1}^\ell\{e_i:i\in I_j\}=\{e_{i_1}+\cdots+e_{i_\ell}:i_j\in I_j\text{ for all }j\in[\ell]\}.\] 
\end{theorem}

In \cite{Murota}, Murota defined a nonempty set $A\subseteq\mathbb Z^{k+1}$ to be \dfn{M-convex} if for all vectors $a=(a_1,\ldots,a_{k+1})$ and $b=(b_1,\ldots,b_{k+1})$ in $A$ and every $i\in[k+1]$ such that $a_i>b_i$, there exists $j\in[k+1]$ with $a_j<b_j$ such that $a-e_i+e_j$ and $b+e_i-e_j$ are both in $A$. A \dfn{discrete polymatroid} (called a \dfn{discrete base polymatroid} in \cite{Herzog}) is a finite M-convex set contained in $\mathbb Z_{\geq 0}^{k+1}$. A polytope is called a \dfn{generalized permutohedron} if each of its edges is parallel to a vector of the form $e_i-e_j$ for $i\neq j$. It is known \cite{Herzog, LamPolypositroids} that a set $A\subseteq\mathbb Z_{\geq 0}^{k+1}$ is a discrete polymatroid if and only if $\conv(A)$ is a generalized permutohedron and $A=\conv(A)\cap\,\mathbb Z^{k+1}$. Herzog and Hibi \cite{Herzog} proved that if $I_1,\ldots,I_\ell$ are nonempty subsets of $[k+1]$, then the set $\{e_{i_1}+\cdots+e_{i_\ell}:i_j\in I_j\text{ for all }j\in[\ell]\}$ is a discrete polymatroid; a discrete polymatroid arising in this way is said to be \dfn{transversal}. According to Theorem~\ref{Thm:PostLattice}, the lattice points of an integral nestohedron form a transversal discrete polymatroid.

Let $f_i(P)$ denote the number of $i$-dimensional faces of a polytope $P$. If $P$ is $d$-dimensional, then $(f_0(P),\ldots,f_d(P))$ is called the \dfn{$f$-vector} of $P$. The \dfn{$f$-polynomial} of $P$ is $f_P(t)=\sum_{i=0}^df_i(P)t^i$. The \dfn{$h$-polynomial} of $P$ is the polynomial $h_P(t)$ defined by the identity $h_P(t)=f_P(t-1)$, and the \dfn{$h$-vector} of $P$ is the vector $(h_0(P),\ldots,h_d(P))$ such that $h_P(t)=\sum_{i=0}^dh_i(P)t^i$. We also make the convention $f_i(P)=h_i(P)=0$ for all $i>d$. 

Now suppose $\Nest({\bf y})$ is a binary nestohedron obtained from a binary building set $\mathcal B$ that, in turn, is obtained from a full binary plane tree $T$. In private communication with the author, Vincent Pilaud has explained how one can easily compute the $f$-vector and $h$-vector of $\Nest({\bf y})$ \cite{PilaudPrivate}. Let $T^{\text c}$ be the plane tree obtained from $T$ by contracting every edge in $T$ whose bottom vertex is not in $\mathcal B$. For each internal vertex $v$ of $T^{\text c}$, let $\deg(v)$ be the number of children of $v$. Then \[f_{\Nest({\bf y})}(t)=\prod_{v}\frac{(t+1)^{\deg(v)}-1}{t}\quad\text{and}\quad h_{\Nest({\bf y})}(t)=\prod_{v}\frac{t^{\deg(v)}-1}{t-1},\] where each product ranges over the set of internal vertices of $T^{\text c}$. For example, if $T$ and $\mathcal B$ are as illustrated in Figure~\ref{Fig1}, then $T^{\text c}$ is \[\begin{array}{l}\includegraphics[height=3.032cm]{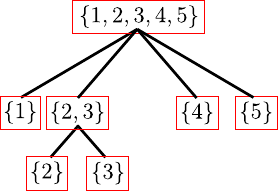}\end{array},\] so \[f_{\Nest({\bf y})}=\frac{(t+1)^4-1}{t}\frac{(t+1)^2-1}{t}=t^4+6t^3+14t^2+16t+8\] and \[h_{\Nest({\bf y})}=\frac{t^4-1}{t-1}\frac{t^2-1}{t-1}=t^4+2t^3+2t^2+2t+1.\]

\begin{remark}\label{Rem2}
Following \cite{Lam}, we say a polytope is \dfn{alcoved} if it is cut out by inequalities of the form $\sum_{i\in[a,b]}x_i\leq r_{a,b}$ for cyclic intervals $[a,b]$ and real numbers $r_{a,b}$. Following the recent article \cite{LamPolypositroids}, we say a polytope is a \dfn{polypositroid} if it is both a generalized permutohedron and an alcoved polytope. Every nestohedron is a generalized permutohedron \cite{Postnikov2}, and it follows from \eqref{Eq:FN} that every binary nestohedron is an alcoved polytope. Therefore, binary nestohedra are polypositroids.
\end{remark}

\section{Fertilitopes are Integral Binary Nestohedra}\label{Sec:Main}

In this section, we prove our main theorems connecting stack-sorting with polytopes and discrete polymatroids. Recall that $\mathcal V(\pi)$ denotes the set of valid compositions of a permutation $\pi$. If $\pi$ is a permutation of size $n$ with $k$ descents, then we can view the valid compositions of $\pi$ as vectors in $\mathbb R^{k+1}$ that lie in the hyperplane $\{(x_1,\ldots,x_{k+1})\in\mathbb R^{k+1}:x_1+\cdots+x_{k+1}=n-k\}$. Define the \dfn{fertilitope} of $\pi$, denoted $\Fer_\pi$, to be the convex hull of the set of valid compositions of $\pi$. That is, $\Fer_\pi=\conv(\mathcal V(\pi))$. 

\begin{theorem}\label{Thm:Main1}
If $\pi$ is a permutation of size $n$ with $k$ descents, then \[\mathcal V(\pi)=\Fer_\pi\cap\,\mathbb Z^{k+1}.\]
\end{theorem}

\begin{theorem}\label{Thm:Main2}
A nonempty polytope is a fertilitope if and only if it is an integral binary nestohedron.  
\end{theorem}

Before proving these theorems, let us discuss their consequences. First, these theorems allow us to view the Refined Tree Fertility Formula (Theorem~\ref{RTFF}) as a sum over lattice points in a polytope; by the discussion following Theorem~\ref{Thm:PostLattice}, this is actually a sum over a transversal discrete polymatroid. In particular, the Refined Fertility Formula (Corollary~\ref{RFF'}) can be rewritten as \[\sum_{\sigma\in s^{-1}(\pi)}x^{\des(\sigma)+1}=\sum_{{\bf q}\in\Fer_\pi\cap\,\mathbb Z^{k+1}}N_{{\bf q}}(x),\] where $\pi$ is a permutation with $k$ descents. Hence, Theorems~\ref{Thm:Main1} and \ref{Thm:Main2} have the following corollary concerning the conjectures stated at the end of Section~\ref{Sec:Intro}. 
\begin{corollary}\label{Cor:Fertilitopes1}
Conjectures~\ref{Conj2} and \ref{Conj:Fertilitopes1} are equivalent. 
\end{corollary}

Specializing further to the Fertility Formula (Corollary~\ref{FF}), we find that \[|s^{-1}(\pi)|=\sum_{{\bf q}\in\Fer_\pi\cap\,\mathbb Z^{k+1}}C_{{\bf q}}.\] This reinterpretation of the Fertility Formula will be crucial when we derive new results about fertility numbers in Section~\ref{Sec:FertilityNumbers}. Indeed, the fact that $\Fer_\pi\cap\,\mathbb Z^{k+1}$ is a discrete polymatroid will be especially helpful in proving that certain numbers are infertility numbers. 

Let us also remark that the VHC Cumulant Formula can now be rewritten as a sum over many transversal discrete polymatroids (possibly with repetitions). Namely, if $(\kappa_n)_{n\geq 1}$ is a sequence of free cumulants, then Theorem~\ref{VHCCF} implies that the corresponding classical cumulants are given by \[-c_n=\sum_{\pi\in S_{n-1}}\sum_{{\bf q}\in\Fer_\pi\cap\,\mathbb Z^{k+1}}(-\kappa_{\bullet+1})_{\bf q}.\] 

Once we have proven Theorem~\ref{Thm:Main2}, it will follow that fertilitopes enjoy all of the nice properties of binary nestohedra discussed in Section~\ref{Subsec:BinaryNestohedra}; this is the primary reason why we listed those properties in the first place. 

\medskip 

We will prove Theorems~\ref{Thm:Main1} and \ref{Thm:Main2} simultaneously by induction. We begin by analyzing the combinatorics of valid hook configurations. Let $\pi=\pi_1\cdots\pi_n$ be a permutation, and let $H$ be a hook of $\pi$ with southwest endpoint $(i,\pi_i)$ and northeast endpoint $(j,\pi_j)$. The \dfn{$H$-unsheltered subpermutation of $\pi$} is the permutation $\pi_U^H=\pi_1\cdots\pi_i\pi_{j+1}\cdots\pi_n$. Similarly, the \dfn{$H$-sheltered subpermutation of $\pi$} is $\pi_S^H=\pi_{i+1}\cdots\pi_{j-1}$. For instance, if $\pi=426315789$ and $H$ is the hook shown in Figure~\ref{Fig4}, then $\pi_U^H=4269$ and $\pi_S^H=3157$. In all applications, the plot of the subpermutation $\pi_S^H$ will lie completely beneath $H$ in the plot of $\pi$ (it will be ``sheltered'' by $H$). 

Now let $\VHC^H(\pi)$ denote the set of valid hook configurations of $\pi$ that include the hook $H$. There are natural maps \[\varphi_U^H:\VHC^H(\pi)\to\VHC(\pi_U^H)\quad\text{and}\quad\varphi_S^H:\VHC^H(\pi)\to\VHC(\pi_S^H).\] Indeed, we obtain $\varphi_U^H(\mathcal H)$ (respectively, $\varphi_S^H(\mathcal H)$) by simply deleting $H$, the northeast endpoint of $H$, and the points and hooks in the plot of $\pi_S^H$ (respectively, $\pi_U^H$) from $\mathcal H$. Figure~\ref{Fig8} provides an example. We define the
\dfn{$H$-splitting map} \[\varphi^H:\VHC^H(\pi)\to\VHC(\pi_U^H)\times\VHC(\pi_S^H)\] by $\varphi^H(\mathcal H)=(\varphi_U^H(\mathcal H),\varphi_S^H(\mathcal H))$. 

\begin{figure}[h]
\begin{center}
\includegraphics[height=7.233cm]{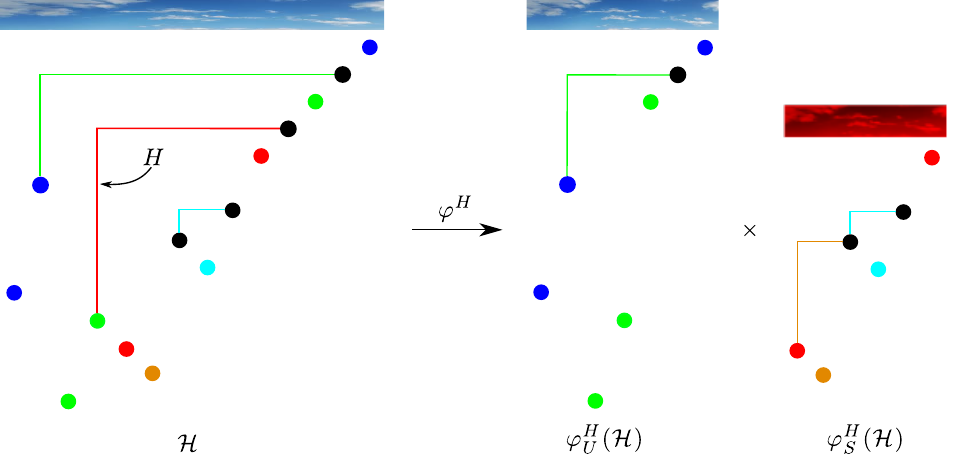}
\caption{An illustration of the maps $\varphi_U^H$, $\varphi_S^H$, and $\varphi^H$ from Proposition~\ref{Prop11}. Notice that the hook $H$ in $\mathcal H$ ``becomes'' the sky in $\varphi_S^H(\mathcal H)$; this is why we colored that sky red instead of the usual color blue. The valid composition induced by $\mathcal H$ is $(3,3,2,1,1)$, and the valid compositions induced by $\varphi_U^H(\mathcal H)$ and $\varphi_S^H(\mathcal H)$ are $(3,3)$ and $(2,1,1)$, respectively.}
\label{Fig8}
\end{center}  
\end{figure}

\begin{proposition}\label{Prop11}
Let $\pi=\pi_1\cdots\pi_n$ be a permutation with descents $d_1<\cdots<d_k$. Let $H$ be a hook of $\pi$ with southwest endpoint $(d_i,\pi_{d_i})$ and northeast endpoint $(j,\pi_j)$, and assume $j>d_k$. The $H$-splitting map $\varphi^H:\VHC^H(\pi)\to\VHC(\pi_U^H)\times\VHC(\pi_S^H)$ is a bijection. If $\mathcal H\in\VHC^H(\pi)$ induces the valid composition $(q_0,\ldots,q_k)$, then the valid compositions induced by $\varphi_U^H(\mathcal H)$ and $\varphi_S^H(\mathcal H)$ are $(q_0,\ldots,q_{i-1})$ and $(q_i,\ldots,q_k)$, respectively. 
\end{proposition}

\begin{proof}
The valid hook configuration $\mathcal H$ can be reconstructed from $\varphi_U^H(\mathcal H)$ and $\varphi_S^H(\mathcal H)$ by piecing these two smaller valid hook configurations together and then adding back the hook $H$ and the northeast endpoint of $H$. To be more precise, let us define $\widetilde\varphi^H:\VHC(\pi_U^H)\times\VHC(\pi_S^H)\to\VHC^H(\pi)$ by declaring the hooks of $\widetilde\varphi^H(\mathcal H_U,\mathcal H_S)$ to be the hooks of $\mathcal H_U$, the hooks of $\mathcal H_S$, and $H$. Here, we are naturally identifying hooks of $\mathcal H_U$ and $\mathcal H_S$ with the hooks of $\pi$ whose endpoints have the same heights. In order to see that $\widetilde\varphi^H$ is the inverse of $\varphi^H$, the only thing we need to check is that $\widetilde\varphi^H$ is well-defined. In other words, we must verify that for all $\mathcal H_U\in\VHC(\pi_U^H)$ and $\mathcal H_S\in\VHC(\pi_S^H)$, the configuration of hooks $\widetilde\varphi^H(\mathcal H_U,\mathcal H_S)$ is a genuine valid hook configuration of $\pi$. It is immediate that the southwest endpoints of the hooks in $\widetilde\varphi^H(\mathcal H_U,\mathcal H_S)$ are the descent tops of the plot of $\pi$. To prove that $\widetilde\varphi^H(\mathcal H_U,\mathcal H_S)$ satisfies  Conditions \ref{Item2} and \ref{Item3} in Definition~\ref{Def5}, the only thing we need to check is that none of the hooks in $\widetilde\varphi^H(\mathcal H_U,\mathcal H_S)$ coming from $\mathcal H_U$ pass underneath points in the plot of $\pi_S^H$, pass underneath the northeast endpoint of $H$, cross hooks coming from $\mathcal H_S$, or cross the hook $H$. This follows from the hypothesis that $j>d_k$. Indeed, this hypothesis guarantees that all of the points in the plot of $\pi$ to the right of the northeast endpoint of $H$ are higher than the northeast endpoint of $H$. The last statement of the proposition concerning the induced valid compositions is immediate from the definition of $\varphi^H$.   
\end{proof}

Let $\pi$ be a permutation whose descents are $d_1<\cdots<d_k$. If $\VHC(\pi)$ is nonempty, then there is a distinguished valid hook configuration $\mathcal H^*=(H_1^*,\ldots,H_k^*)$ called the \dfn{canonical hook configuration} of $\pi$. This is the valid hook configuration of $\pi$ obtained by choosing the northeast endpoints of all of the hooks to be as low as possible. To be more precise, let us say a point in the plot of $\pi$ lies \dfn{weakly below} a hook $H$ if it lies below $H$ or is the northeast endpoint of $H$ (the southwest endpoint of $H$ \emph{does not} lie weakly below $H$). We construct the hooks in the canonical hook configuration in the order $H_k^*,\ldots,H_1^*$. Choose $H_\ell^*$ to be the hook with southwest endpoint $(d_\ell,\pi_{d_\ell})$ whose northeast endpoint is the lowest point in the plot of $\pi$ that lies above and to the right of $(d_\ell,\pi_{d_\ell})$ and does not lie weakly below any of the hooks $H_k^*,\ldots,H_{\ell+1}^*$ that have already been constructed. If at any time the hook $H_\ell^*$ does not exist, then $\pi$ does not have a canonical hook configuration. It is not difficult to show that every permutation that has a valid hook configuration must have a canonical hook configuration. 

We need just a little more notation before stating our first key lemma. Recall that for $\pi=\pi_1\cdots\pi_n\in S_n$, we let $\pi\oplus 1$ be the concatenation $\pi(n+1)\in S_{n+1}$. The \dfn{tail length} of $\pi$, denoted $\tl(\pi)$, is the largest integer $\ell\in\{0,\ldots,n\}$ such that $\pi_i=i$ for all $i\in\{n-\ell+1,\ldots,n\}$. For example, $\tl(13245)=2$, $\tl(13254)=0$, and $\tl(12345)=5$. The \dfn{tail} of $\pi$ is the set of points $\{(n-\tl(\pi)+1,n-\tl(\pi)+1),\ldots,(n,n)\}$. Recall that a permutation is  sorted (i.e., in the image of $s$) if and only if $\mathcal V(\pi)\neq\emptyset$. 

\begin{lemma}\label{LemFertilitopes1}
If $\pi\in S_n$ is a sorted permutation with $k$ descents, then \[\mathcal V(\pi\oplus 1)\subseteq \mathcal V(\pi)+\{e_1,\ldots,e_{k+1}\}.\]
\end{lemma}

\begin{proof}
If $\pi=\id_n$, then $k=0$, $\mathcal V(\pi)=\{(n)\}$, and $\mathcal V(\pi\oplus 1)=\mathcal V(\id_{n+1})=\{(n+1)\}=\{(n)\}+\{e_1\}$. This completes the proof when $n=\id_n$; in particular, this finishes the proof when $n=1$. Thus, we may assume $\pi\neq\id_n$ and $n\geq 2$ and proceed by induction on $n$. Let $d_1<\cdots <d_k$ be the descents of $\pi$; observe that these are also the descents of $\pi\oplus 1$. By the definition of the tail length, there is an index $m\in[k]$ such that $\pi_{d_m}=n-\tl(\pi)$. 

Since $\mathcal V(\pi)\neq\emptyset$, the permutation $\pi$ must have a canonical hook configuration $\mathcal H^*=(H_1^*,\ldots,H_k^*)$. By the definition of the index $m$, we know that $H_m^*$ (which has southwest endpoint $(d_m,\pi_{d_m})$) has a northeast endpoint that lies in the tail of $\pi$. This means that the northeast endpoint of $H_m^*$ is $(b,b)$ for some integer $b$ with $n-\tl(\pi)+1\leq b\leq n$. 

Choose a valid hook configuration $\mathcal H=(H_1,\ldots,H_k)\in\VHC(\pi\oplus 1)$, and let ${\bf q}^{\mathcal H}=(q_0,\ldots,q_k)$ be the valid composition of $\pi\oplus 1$ that $\mathcal H$ induces. The hook $H_m$ has southwest endpoint $(d_m,\pi_{d_m})$ and northeast endpoint $(j,j)$ for some integer $j$ with $n-\tl(\pi)+1\leq j\leq n+1$. To prove that ${\bf q}^{\mathcal H}\in\mathcal V(\pi)+\{e_1,\ldots,e_{k+1}\}$, we consider two cases depending on whether or not $j=b$.

\medskip
\noindent {\bf Case 1.} Suppose $j=b$. This means that the hook $H_m$ of $\pi\oplus 1$ has the same endpoints as the hook $H_m^*$ of $\pi$. Therefore, $\pi_S^{H_m^*}=(\pi\oplus 1)_S^{H_m}$. Let $\sigma$ be the standardization of $\pi_U^{H_m^*}$, and observe that $\sigma\oplus 1$ is the standardization of $(\pi\oplus 1)_U^{H_m}$. Since two permutations with the same relative order have the same set of valid compositions, we have $\mathcal V\left((\pi\oplus 1)_U^{H_m}\right)=\mathcal V(\sigma\oplus 1)$ and $\mathcal V\left(\pi_U^{H_m^*}\right)=\mathcal V(\sigma)$. Proposition~\ref{Prop11} tells us that the valid composition of $(\pi\oplus 1)_U^{H_m}$ induced by $\varphi_U^{H_m}(\mathcal H)$ is $(q_0,\ldots,q_{m-1})$ and that the valid composition of $(\pi\oplus 1)_S^{H_m}$ induced by $\varphi_S^{H_m}(\mathcal H)$ is $(q_m,\ldots,q_k)$. This implies that $\mathcal V(\sigma)\neq\emptyset$, so we can use induction on $n$ to see that \[\mathcal V\left((\pi\oplus 1)_U^{H_m}\right)=\mathcal V(\sigma\oplus 1)\subseteq\mathcal V(\sigma)+\{e_1,\ldots,e_m\}=\mathcal V\left(\pi_U^{H_m^*}\right)+\{e_1,\ldots,e_m\}.\] It follows that $(q_0,\ldots,q_{m-1})=(q_0',\ldots,q_{m-1}')+e_i$ for some $(q_0',\ldots,q_{m-1}')\in\mathcal V\left(\pi_U^{H_m^*}\right)$ and $i\in[m]$ and that $(q_m,\ldots,q_k)\in\mathcal V\left((\pi\oplus 1)_S^{H_m}\right)=\mathcal V\left(\pi_S^{H_m^*}\right)$. Let $\mathcal H_U'$ be the valid hook configuration of $\pi_U^{H_m^*}$ that induces the valid composition $(q_0',\ldots,q_{m-1}')$, and let $\mathcal H_S'$ be the valid hook configuration of $\pi_S^{H_m^*}$ that induces the valid composition $(q_m,\ldots,q_k)$. According to Proposition~\ref{Prop11}, there is a valid hook configuration $\mathcal H'\in\VHC^{H_m^*}(\pi)$ such that $\varphi^{H_m^*}(\mathcal H')=(\mathcal H_U',\mathcal H_S')$. Furthermore, the valid composition of $\pi$ induced by $\mathcal H'$ is $(q_0',\ldots,q_{m-1}',q_m,\ldots,q_k)={\bf q}^{\mathcal H}-e_i$. Thus, ${\bf q}^{\mathcal H}\in\mathcal V(\pi)+\{e_1,\ldots,e_{k+1}\}$. 

\medskip
\noindent {\bf Case 2.} Suppose $j\neq b$. The hooks in the canonical hook configuration of $\pi\oplus 1$ are the same as the hooks $H_1^*,\ldots,H_k^*$ in the canonical hook configuration $\mathcal H^*$ of $\pi$. Since the canonical hook configuration of a permutation is constructed by choosing the northeast endpoints of the hooks as low as possible, the northeast endpoint $(j,j)$ of $H_m$ cannot be lower than the northeast endpoint $(b,b)$ of $H_m^*$. That is, $j>b$. Observe that $H_{m+1}^*,\ldots,H_k^*$ are precisely the hooks in $\mathcal H^*$ that lie below $H_m^*$. 

We now construct a new valid hook configuration $\widetilde{\mathcal H}=\left(\widetilde{H}_1,\ldots,\widetilde{H}_k\right)\in\VHC(\pi)$. For $m+1\leq i\leq k$, let $\widetilde{H}_i=H_i^*$. If $1\leq i\leq m$ and the northeast endpoint of $H_i$ is a point $(t,t)$ in the tail of $\pi\oplus 1$, let $\widetilde{H}_i$ be the hook with southwest endpoint $(d_i,\pi_{d_i})$ and northeast endpoint $(t-1,t-1)$. In particular, $\widetilde{H}_m$ has northeast endpoint $(j-1,j-1)$. Finally, if $1\leq i\leq m$ and the northeast endpoint of $H_i$ is not in the tail of $\pi\oplus 1$, let $\widetilde{H}_i$ be the hook of $\pi$ with the same endpoints as the hook $H_i$ of $\pi\oplus 1$. 

Let $\sigma$ be the standardization of $\pi_S^{\widetilde{H}_m}$. Since $\varphi_S^{\widetilde{H}_m}(\widetilde{\mathcal H})$ is a valid hook configuration of $\pi_S^{\widetilde{H}_m}$, the set $\mathcal V\left(\pi_S^{\widetilde{H}_m}\right)$ is nonempty. Thus, $\mathcal V(\sigma)\neq\emptyset$. Now observe that $(\pi\oplus 1)_S^{H_m}$ has the same relative order as $\sigma\oplus 1$. Since two permutations with the same relative order have the same set of valid compositions, we can use induction on $n$ to see that \[\mathcal V\left((\pi\oplus 1)_S^{H_m}\right)=\mathcal V(\sigma\oplus 1)\subseteq\mathcal V(\sigma)+\{e_1,\ldots,e_{k+1-m}\}=\mathcal V\left(\pi_S^{\widetilde{H}_m}\right)+\{e_1,\ldots,e_{k+1-m}\}.\] Also, $(\pi\oplus 1)_U^{H_m}$ has the same relative order as $\pi_U^{\widetilde{H}_m}$, so $\mathcal V\left((\pi\oplus 1)_U^{H_m}\right)=\mathcal V\left(\pi_U^{\widetilde{H}_m}\right)$.
Proposition~\ref{Prop11} tells us that the valid composition of $(\pi\oplus 1)_U^{H_m}$ induced by $\varphi_U^{H_m}(\mathcal H)$ is $(q_0,\ldots,q_{m-1})$ and that the valid composition of $(\pi\oplus 1)_S^{H_m}$ induced by $\varphi_S^{H_m}(\mathcal H)$ is $(q_m,\ldots,q_k)$. It follows that $(q_0,\ldots,q_{m-1})\in\mathcal V\left((\pi\oplus 1)_U^{H_m}\right)=\mathcal V\left(\pi_U^{\widetilde{H}_m}\right)$ and that $(q_m,\ldots,q_k)=(q_m',\ldots,q_k')+e_i$ for some $(q_m',\ldots,q_k')\in\mathcal V\left(\pi_S^{\widetilde{H}_m}\right)$ and $i\in[k+1-m]$. Let $\mathcal H_U'$ be the valid hook configuration of $\pi_U^{\widetilde{H}_m}$ that induces the valid composition $(q_0,\ldots,q_{m-1})$, and let $\mathcal H_S'$ be the valid hook configuration of $\pi_S^{\widetilde{H}_m}$ that induces the valid composition $(q_m',\ldots,q_k')$. According to Proposition~\ref{Prop11}, there is a valid hook configuration $\mathcal H'\in\VHC^{\widetilde{H}_m}(\pi)$ such that $\varphi^{\widetilde{H}_m}(\mathcal H')=(\mathcal H_U',\mathcal H_S')$. Furthermore, the valid composition of $\pi$ induced by $\mathcal H'$ is $(q_0,\ldots,q_{m-1},q_m',\ldots,q_k')={\bf q}^{\mathcal H}-e_{m+i}$. Thus, ${\bf q}^{\mathcal H}\in\mathcal V(\pi)+\{e_1,\ldots,e_{k+1}\}$ in this case as well. 
\end{proof}

\begin{figure}[ht]
  \begin{center}{\includegraphics[height=11.94cm]{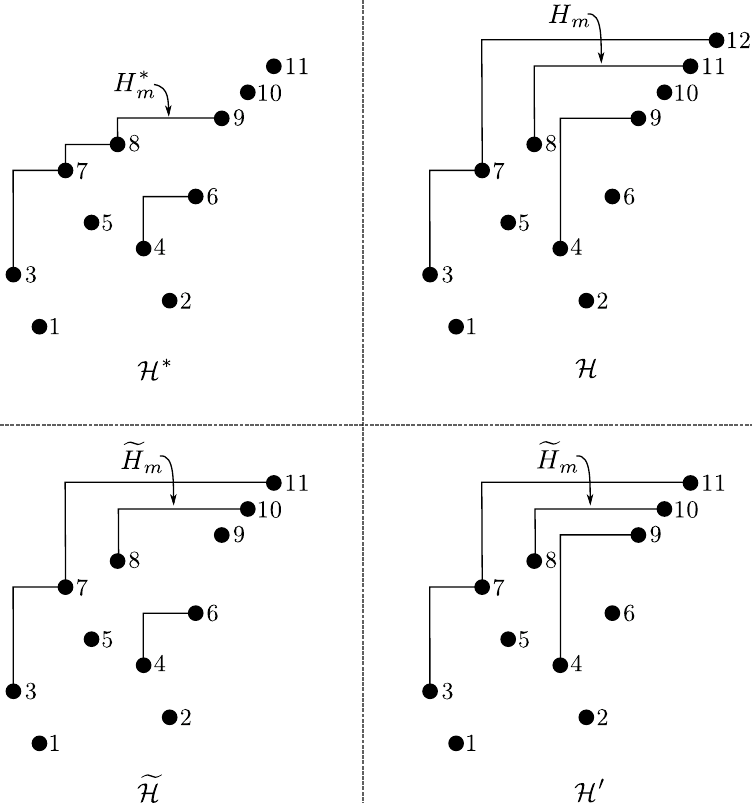}}
  \end{center}
  \caption{Valid hook configurations from Case 2 in the proof of Lemma~\ref{LemFertilitopes1}.}\label{FigFertilitopes2}
\end{figure}

\begin{example}
Let us illustrate Case 2 from the proof of Lemma~\ref{LemFertilitopes1}. Let $\pi=3\,1\,7\,5\,8\,4\,2\,6\,9\,10\,11$. The canonical hook configuration $\mathcal H^*$ of $\pi$ is presented in the upper left panel of Figure~\ref{FigFertilitopes2}. In this example, $m=3$ and $b=9$. Let us take $\mathcal H$ to be the valid hook configuration of $\pi\oplus 1$ in the upper right panel of Figure~\ref{FigFertilitopes2}. We have $j=11>b$, so we are indeed in Case 2 of the proof. The valid hook configuration $\widetilde{\mathcal H}$ appears in the bottom left panel of Figure~\ref{FigFertilitopes2}. We have ${\bf q}^{\mathcal H}=(q_0,\ldots,q_4)=(1,1,2,2,2)$. Also, $\pi_U^{\widetilde H_m}=3\,1\,7\,5\,8\,11$ and $\pi_S^{\widetilde H_m}=4269$. The valid hook configuration $\mathcal H_U'$ of $\pi_U^{\widetilde H_m}$ that induces the composition $(q_0,\ldots,q_{m-1})=(1,1,2)$ has one hook whose endpoints have heights $3$ and $7$ and a second hook whose endpoints have heights $7$ and $11$. In order to construct $\mathcal H_S'$, we must find an index $i\in\{1,2\}$ such that $(2,2)=(q_3',q_4')+e_i$ for some $(q_3',q_4')\in\mathcal V\left(\pi_S^{\widetilde{H}_m}\right)$ (the existence of such an index is ensured by the induction hypothesis in the proof). In this specific example, we have $\mathcal V\left(\pi_S^{\widetilde{H}_m}\right)=\{(1,2),(2,1)\}$, so we could choose either $i=1$ or $i=2$. Let us choose $i=1$. Then $\mathcal H_S'$ has a single hook whose endpoints have heights $4$ and $9$. The valid hook configuration $\mathcal H'=\left(\varphi^{\widetilde{H}_m}\right)^{-1}(\mathcal H_U',\mathcal H_S')$ is shown in the bottom right panel of Figure~\ref{FigFertilitopes2}; note that it induces the valid composition $(1,1,2,1,2)=(1,1,2,2,2)-e_4={\bf q}^{\mathcal H}-e_{m+i}$.  
\end{example}

\begin{lemma}\label{Lem2}
If $\pi=\pi_1\cdots\pi_n\in S_n$ is sorted, then \[\mathcal V(\pi\oplus 1)=\mathcal V(\pi)+\{e_1,\ldots,e_{k+1}\}.\]
\end{lemma}

\begin{proof}
In light of Lemma~\ref{LemFertilitopes1}, it suffices to prove that $\mathcal V(\pi\oplus 1)\supseteq\mathcal V(\pi)+\{e_1,\ldots,e_{k+1}\}$. If $\pi=\id_n$, then $k=0$, $\mathcal V(\pi)=\{(n)\}$, and $\mathcal V(\pi\oplus 1)=\mathcal V(\id_{n+1})=\{(n+1)\}=\{(n)\}+\{e_1\}$. Therefore, we may assume $\pi\neq\id_n$ and $n\geq 2$ and proceed by induction on $n$. Let $d_1<\cdots<d_k$ be the descents of $\pi$ (these are also the descents of $\pi\oplus 1$). By the definition of the tail length, there is an index $m\in[k]$ such that $\pi_{d_m}=n-\tl(\pi)$. Observe that every hook of $\pi$ with southwest endpoint $(d_m,\pi_{d_m})$ has a northeast endpoint that lies in the tail of $\pi$. 

Choose a composition ${\bf q}\in\mathcal V(\pi)+\{e_1,\ldots,e_{k+1}\}$. There is an index $i\in[k+1]$ such that ${\bf q}=(q_0,\ldots,q_k)+e_i$ for some $(q_0,\ldots,q_k)\in\mathcal V(\pi)$. There exists a valid hook configuration $\mathcal H=(H_1,\ldots,H_k)$ of $\pi$ such that $(q_0,\ldots,q_k)$ is the valid composition ${\bf q}^{\mathcal H}$ induced by $\mathcal H$. According to Proposition~\ref{Prop11}, the valid composition of $\pi_U^{H_m}$ induced by $\varphi_U^{H_m}(\mathcal H)$ is $(q_0,\ldots,q_{m-1})$. Similarly, the valid composition of $\pi_S^{H_m}$ induced by $\varphi_S^{H_m}(\mathcal H)$ is $(q_m,\ldots,q_k)$. Our goal is to show that ${\bf q}$ is a valid composition of $\pi\oplus 1$; we consider two cases based on whether $i\leq m$ or $i\geq m+1$. 

\medskip
\noindent {\bf Case 1.} Suppose $i\leq m$. We can view $H_m$ as a hook of $\pi\oplus 1$ and consider the $H_m$-unsheltered subpermutation $(\pi\oplus 1)_U^{H_m}$. Let $\sigma$ be the standardization of $\pi_U^{H_m}$, and observe that the standardization of $(\pi\oplus 1)_U^{H_m}$ has the same relative order as $\sigma\oplus 1$. We have $\mathcal V\left(\pi_U^{H_m}\right)=\mathcal V(\sigma)$ and $\mathcal V\left((\pi\oplus 1)_U^{H_m}\right)=\mathcal V(\sigma\oplus 1)$. Since $(q_0,\ldots,q_{m-1})\in\mathcal V\left(\pi_U^{H_m}\right)=\mathcal V(\sigma)$, the permutation $\sigma$ must be sorted. Hence, we can use induction on $n$ to see that $\mathcal V(\sigma\oplus 1)=\mathcal V(\sigma)+\{e_1,\ldots,e_m\}$ and, therefore, $(q_0,\ldots,q_{m-1})+e_i\in\mathcal V(\sigma\oplus 1)=\mathcal V\left((\pi\oplus 1)_U^{H_m}\right)$. We also have $(\pi\oplus 1)_S^{H_m}=\pi_S^{H_m}$, so $(q_m,\ldots,q_k)\in\mathcal V\left(\pi_S^{H_m}\right)=\mathcal V\left((\pi\oplus 1)_S^{H_m}\right)$. Let $\mathcal H_U'$ and $\mathcal H_S'$ be the valid hook configurations of $(\pi\oplus 1)_U^{H_m}$ and $(\pi\oplus 1)_S^{H_m}$ that induce the valid compositions $(q_0,\ldots,q_{m-1})+e_i$ and $(q_m,\ldots,q_k)$, respectively. Proposition~\ref{Prop11} tells us that the valid hook configuration $\left(\varphi^{H_m}\right)^{-1}(\mathcal H_U',\mathcal H_S')$ of $\pi\oplus 1$ induces the composition $(q_0,\ldots,q_k)+e_i={\bf q}$, as desired.

\medskip
\noindent {\bf Case 2.} Suppose $i\geq m$. Let $(b,b)$ be the northeast endpoint of $H_m$. Let $H_m'$ be the hook of $\pi\oplus 1$ with southwest endpoint $(d_m,\pi_{d_m})$ and northeast endpoint $(b+1,b+1)$. Let $\sigma$ be the standardization of $\pi_S^{H_m}$, and observe that the standardization of $(\pi\oplus 1)_S^{H_m'}$ has the same relative order as $\sigma\oplus 1$. We have $\mathcal V\left(\pi_S^{H_m}\right)=\mathcal V(\sigma)$ and $\mathcal V\left((\pi\oplus 1)_S^{H_m'}\right)=\mathcal V(\sigma\oplus 1)$. Since $(q_m,\ldots,q_k)\in\mathcal V\left(\pi_S^{H_m}\right)=\mathcal V(\sigma)$, the permutation $\sigma$ is sorted. By induction, $\mathcal V(\sigma\oplus 1)=\mathcal V(\sigma)+\{e_1,\ldots,e_{k+1-m}\}$ and, therefore, \[(q_m,\ldots,q_k)+e_{i-m}\in\mathcal V(\sigma\oplus 1)=\mathcal V\left((\pi\oplus 1)_S^{H_m'}\right).\] Now observe that $(\pi\oplus 1)_U^{H_m'}$ has the same relative order as $\pi_U^{H_m}$; this implies that $(q_0,\ldots,q_{m-1})\in\mathcal V\left(\pi_U^{H_m}\right)=\mathcal V\left((\pi\oplus 1)_U^{H_m'}\right)$. Let $\mathcal H_U'$ and $\mathcal H_S'$ be the valid hook configurations of $(\pi\oplus 1)_U^{H_m'}$ and $(\pi\oplus 1)_S^{H_m'}$ that induce the valid compositions $(q_0,\ldots,q_{m-1})$ and $(q_m,\ldots,q_k)+e_{i-m}$, respectively. Proposition~\ref{Prop11} tells us that the valid hook configuration $\left(\varphi^{H_m'}\right)^{-1}(\mathcal H_U',\mathcal H_S')$ of $\pi\oplus 1$ induces the composition $(q_0,\ldots,q_k)+e_i={\bf q}$, as desired.
\end{proof}

The previous lemma allows us to make the first step toward connecting fertilitopes with binary nestohedra. 

\begin{proposition}\label{Prop1}
Let $\pi\in S_n$ be a permutation with $k$ descents such that $\Fer_\pi$ is a binary nestohedron and $\mathcal V(\pi)=\Fer_\pi\cap\,\mathbb Z^{k+1}\neq\emptyset$. Then $\Fer_{\pi\oplus 1}$ is a binary nestohedron such that \[\Fer_{\pi\oplus 1}=\Fer_\pi+\Delta_{[k+1]}\quad\text{and}\quad\mathcal V(\pi\oplus 1)=\Fer_{\pi\oplus 1}\cap\,\mathbb Z^{k+1}.\]
\end{proposition}

\begin{proof}
According to Lemma~\ref{Lem2}, we have \[\Fer_{\pi\oplus 1}=\conv(\mathcal V(\pi\oplus 1))=\conv\left(\mathcal V(\pi)+\{e_1,\ldots,e_{k+1}\}\right)=\Fer_\pi+\Delta_{[k+1]}.\] Because $\Fer_\pi$ is a binary nestohedron, there exist a binary building set $\mathcal B$ on $[k+1]$ and a tuple ${\bf y}=(y_I)_{I\in\mathcal B}\in\mathbb R_{>0}^{\mathcal B}$ such that $\Fer_\pi=\Nest({\bf y})$. Let $\mathcal B'=\mathcal B\cup\{[k+1]\}$. For $I\in\mathcal B'\setminus\{[k+1]\}$, let $y_I'=y_I$. If $[k+1]\in\mathcal B$, let $y_{[k+1]}'=y_{[k+1]}+1$; otherwise, let $y_{[k+1]}'=1$. Then $\Fer_{\pi\oplus 1}=\Nest({\bf y}')$, where ${\bf y}'=(y_I')_{I\in\mathcal B'}\in\mathbb R_{>0}^{\mathcal B'}$. Since $\mathcal B'$ is a binary building set on $[k+1]$, the fertilitope $\pi\oplus 1$ is a binary nestohedron. 

Fertilitopes are integral by definition, so Corollary~\ref{Cor1} tells us that the numbers $y_I$ for $I\in\mathcal B$ are all positive integers. Let $I_1,\ldots,I_\ell$ be a list of the sets in $\mathcal B$ such that each set $I$ appears exactly $y_I$ times. Then $\Fer_\pi=\Nest({\bf y})=\sum_{j=1}^\ell\Delta_{I_j}$, so we can use Theorem~\ref{Thm:PostLattice} to see that $\mathcal V(\pi)=\Fer_\pi\cap\,\mathbb Z^{k+1}=\sum_{j=1}^\ell\{e_i:i\in I_j\}$. Similarly, $\left(\left(\sum_{j=1}^\ell\Delta_{I_j}\right)+\Delta_{[k+1]}\right)\cap\,\mathbb Z^{k+1}=\left(\sum_{j=1}^\ell\{e_i:i\in I_j\}\right)+\{e_1,\ldots,e_{k+1}\}=\mathcal V(\pi)+\{e_1,\ldots,e_{k+1}\}$. Invoking Lemma~\ref{Lem2}, we find that \[\mathcal V(\pi\oplus 1)=\mathcal V(\pi)+\{e_1,\ldots,e_{k+1}\}=\left(\left(\sum_{j=1}^\ell\Delta_{I_j}\right)+\Delta_{[k+1]}\right)\cap\,\mathbb Z^{k+1}=\Fer_{\pi\oplus 1}\cap\,\mathbb Z^{k+1}.\qedhere\]  
\end{proof}

For ease of exposition, we isolate the proof of one direction of Theorem~\ref{Thm:Main2} in a lemma. 

\begin{lemma}\label{Lem:Fertilitopes3}
Every integral binary nestohedron is a fertilitope. 
\end{lemma}

\begin{proof}
Let $\Nest({\bf y})$ be an integral binary nestohedron, where $\mathcal B$ is a binary building set on $[k+1]$ and ${\bf y}=(y_I)_{I\in\mathcal B}\in\mathbb R_{>0}^{\mathcal B}$. Corollary~\ref{Cor1} tells us that the numbers $y_I$ are actually positive integers. We will prove by induction on the sum $Y=\sum_{I\in\mathcal B}y_I$ that $\Nest({\bf y})$ is a fertilitope. If $Y=1$, then we have $k=0$, $\mathcal B=\{\{1\}\}$, and $y_{\{1\}}=1$, so $\Nest({\bf y})=\{(1)\}\subseteq\mathbb R$ is the fertilitope of the permutation $1\in S_1$. Now assume $Y\geq 2$. 

Suppose $[k+1]\in\mathcal B$. Let $y_I'=y_I$ for all $I\in\mathcal B\setminus\{[k+1]\}$ and $y_{[k+1]}'=y_{[k+1]}-1$. The Minkowski sum $P'=\sum_{I\in\mathcal B}y_I'\Delta_{I}$ is a binary nestohedron, so we can use induction on $Y$ to see that there exist a positive integer $n$ and a permutation $\tau\in S_n$ such that $P'=\Fer_\tau$. According to Proposition~\ref{Prop1}, we have $\Nest({\bf y})=P'+\Delta_{[k+1]}=\Fer_\tau+\Delta_{[k+1]}=\Fer_{\tau\oplus 1}$, so $\Nest({\bf y})$ is a fertilitope. 

We now assume $[k+1]\not\in\mathcal B$. It follows immediately from the definition of a binary nestohedron that there exist (nonempty) binary nestohedra $Q$ and $Q'$ such that $\Nest({\bf y})=Q\times Q'$. By induction, there are positive integers $n$ and $n'$ and permutations $\tau=\tau_1\cdots\tau_n\in S_n$ and $\tau'=\tau_1'\cdots\tau_{n'}'\in S_{n'}$ such that $Q=\Fer_\tau$ and $Q'=\Fer_{\tau'}$. Since $\mathcal V(\tau)$ and $\mathcal V(\tau')$ are nonempty, we know by Remark~\ref{Rem1} that $\tau_n=n$ and $\tau_{n'}=n'$. Consider the permutation $\sigma=(\tau_1+n')\cdots (\tau_n+n')\tau_1'\cdots\tau_{n'}'(n+n'+1)\in S_{n+n'+1}$. Let $H$ be the hook of $\sigma$ with southwest endpoint $(n,n+n')$ and northeast endpoint $(n+n'+1,n+n'+1)$; this is the unique hook of $\sigma$ with southwest endpoint $(n,n+n')$. Because $n$ is a descent of $\sigma$, every valid hook configuration of $\sigma$ must contain $H$. That is, $\VHC(\sigma)=\VHC^H(\sigma)$. Proposition~\ref{Prop11} now tells us that the $H$-splitting map $\varphi^{H}$ is a bijection from $\VHC(\sigma)$ to $\VHC\left(\sigma_U^{H}\right)\times \VHC\left(\sigma_S^{H}\right)$. By Proposition~\ref{Prop11} and Remark~\ref{Rem6}, we have $\mathcal V(\sigma)=\mathcal V\left(\sigma_U^{H}\right)\times\mathcal V\left(\sigma_S^{H}\right)$. The subpermutations $\sigma_U^H$ and $\sigma_S^H$ have the same relative orders as $\tau$ and $\tau'$, respectively, so $\mathcal V\left(\sigma\right)=\mathcal V(\tau)\times\mathcal V(\tau')$. Finally, \[\Nest({\bf y})=\Fer_\tau\times\Fer_{\tau'}=\conv(\mathcal V(\tau))\times\conv(\mathcal V(\tau'))=\conv(\mathcal V(\tau)\times\mathcal V(\tau'))=\conv(\mathcal V(\sigma))=\Fer_\sigma,\] so $\Nest({\bf y})$ is a fertilitope.  
\end{proof}

We are now in a position to complete the proofs of our main theorems. 

\begin{proof}[Proof of Theorems~\ref{Thm:Main1} and \ref{Thm:Main2}]
Let $\pi$ be a sorted permutation of size $n$ with $k$ descents. We will prove by induction on $n$ that $\Fer_\pi$ is a binary nestohedron (it is automatically integral by definition) such that $\mathcal V(\pi)=\Fer_\pi\cap\,\mathbb Z^{k+1}$. Since permutations with the same relative order have the same fertilitope, we may assume $\pi\in S_n$. If $\pi=\id_n$, then $\Fer_\pi=\mathcal V(\pi)=\{(n)\}$ is a binary nestohedron corresponding to the binary building set $\mathcal B=\{\{1\}\}$ with the coefficient $y_{\{1\}}=n$. Thus, we may assume $n\geq 2$ and $\pi\neq\id_n$. 

Because $\pi$ is sorted (meaning $\mathcal V(\pi)\neq\emptyset$), it has a canonical hook configuration $\mathcal H^*=(H_1^*,\ldots,H_k^*)$. It follows from the definition of $s$ that $\pi_n=n$. Thus, we may write $\pi=\pi'\oplus 1$ for some $\pi'\in S_{n-1}$. If $\mathcal V(\pi')\neq\emptyset$, then we know by induction that $\Fer_{\pi'}$ is a binary nestohedron and that $\mathcal V(\pi')=\Fer_{\pi'}\cap\,\mathbb Z^{k+1}$. Therefore, the desired result follows immediately from Proposition~\ref{Prop1} in this case.  

We now assume $\mathcal V(\pi')=\emptyset$. Equivalently, $\pi'$ has no valid hook configurations. If none of the hooks in $\mathcal H^*$ had the point $(n,n)$ as a northeast endpoint, then we could view $\mathcal H^*$ as a valid hook configuration of $\pi'$, which would be a contradiction. Therefore, there must be some hook $H_d^*$ in $\mathcal H^*$ with northeast endpoint $(n,n)$. Let $\mathcal H=(H_1,\ldots,H_k)$ be an arbitrary valid hook configuration of $\pi$. The canonical hook configuration of $\pi$ is constructed by choosing the northeast endpoints of the hooks as low as possible. This implies that the northeast endpoint of $H_d$ must be at least as high as that of $H_d^*$, so the northeast endpoint of $H_d$ must also be $(n,n)$. Since $H_d$ and $H_d^*$ have the same southwest endpoint $(d,\pi_d)$, we conclude that $H_d=H_d^*$. This proves that the set $\VHC(\pi)$ of valid hook configurations of $\pi$ is equal to the set $\VHC^{H_d^*}(\pi)$ of valid hook configurations of $\pi$ that contain the hook $H_d^*$. Therefore, Proposition~\ref{Prop11} tells us that the $H_d^*$-splitting map $\varphi^{H_d^*}$ is a bijection from $\VHC(\pi)$ to $\VHC\left(\pi_U^{H_d^*}\right)\times \VHC\left(\pi_S^{H_d^*}\right)$. It now follows from Proposition~\ref{Prop11} and Remark~\ref{Rem6} that $\mathcal V(\pi)=\mathcal V\left(\pi_U^{H_d^*}\right)\times\mathcal V\left(\pi_S^{H_d^*}\right)$. Hence, \[\Fer_\pi=\conv(\mathcal V(\pi))=\conv\left(\mathcal V\left(\pi_U^{H_d^*}\right)\times\mathcal V\left(\pi_S^{H_d^*}\right)\right)=\Fer_{\pi_U^{H_d^*}}\times\Fer_{\pi_S^{H_d^*}}.\]

We know by induction on $n$ that $\Fer_{\pi_U^{H_d^*}}$ and $\Fer_{\pi_S^{H_d^*}}$ are binary nestohedra such that \[\mathcal V\left(\pi_U^{H_d^*}\right)=\Fer_{\pi_U^{H_d^*}}\cap\,\mathbb Z^{k+1}\quad\text{and}\quad\mathcal V\left(\pi_S^{H_d^*}\right)=\Fer_{\pi_S^{H_d^*}}\cap\,\mathbb Z^{k+1}.\] It is straightforward to check that the Cartesian product of two binary nestohedra is a binary nestohedron, so we conclude that $\Fer_\pi$ is a binary nestohedron. Moreover, \[\mathcal V(\pi)=\left(\Fer_{\pi_U^{H_d^*}}\cap\,\mathbb Z^{k+1}\right)\times\left(\Fer_{\pi_S^{H_d^*}}\cap\,\mathbb Z^{k+1}\right)=\left(\Fer_{\pi_U^{H_d^*}}\times \Fer_{\pi_S^{H_d^*}}\right)\cap\,\mathbb Z^{k+1}=\Fer_\pi\cap\,\mathbb Z^{k+1}.\]  

This proves Theorem~\ref{Thm:Main1} and one direction of Theorem~\ref{Thm:Main2}; the other direction of Theorem~\ref{Thm:Main2} is Lemma~\ref{Lem:Fertilitopes3}. 
\end{proof}

We end this section with a consequence of the main theorems that we have just proven. For $\pi\in S_n$ and $m\geq 0$, recall that we write $\pi\oplus \id_m$ for the permutation $\pi(n+1)\cdots(n+m)\in S_{n+m}$. If $\pi$ is a sorted permutation with $k$ descents, then Theorem~\ref{Thm:Main1}, Theorem~\ref{Thm:Main2}, and Proposition~\ref{Prop1} imply that $\Fer_{\pi\oplus \id_m}=\Fer_\pi+m\Delta_{[k+1]}$ and that $\mathcal V(\pi\oplus \id_m)=\left(\Fer_\pi+m\Delta_{[k+1]}\right)\cap\,\mathbb Z^{k+1}$, where $\Fer_\pi$ is an integral binary nestohedron. The Ehrhart-style function $|\left(\Fer_\pi+m\Delta_{[k+1]}\right)\cap\mathbb Z^{k+1}|$ is known to be a degree-$k$ polynomial in $m$ (see, e.g., \cite{Kalman}). Hence, we have the following corollary. 

\begin{corollary}
For every sorted permutation $\pi\in S_n$ with $\des(\pi)=k$, there exists a polynomial $g_\pi(t)\in\mathbb R[t]$ of degree $k$ such that $|\mathcal V(\pi\oplus\id_m)|=g_\pi(m)$ for all $m\geq 0$.  
\end{corollary}

\section{Canonical and Quasicanonical Trees}\label{Sec:Quasicanonical}

In this section, we provide a method for explicitly computing, for each sorted permutation $\pi$, a tuple ${\bf y}^\pi$ such that $\Fer_\pi=\Nest({\bf y}^\pi)$. In order to do so, we introduce quasicanonical trees, which we prove are in bijection with valid hook configurations. Incidentally, we will obtain a new combinatorial formula for converting from a sequence of free cumulants to the corresponding sequence of classical cumulants. 

Bousquet-M\'elou defined a decreasing binary plane tree to be \dfn{canonical} if every vertex $v$ that has a left child also has a nonempty right subtree $\mathcal T_v^R$ such that the first entry in the in-order traversal $\mathcal I(\mathcal T_v^R)$ is smaller than the label of the left child of $v$. Let us say a decreasing binary plane tree is \dfn{quasicanonical} if every vertex $v$ that has a left child also has a nonempty right subtree $\mathcal T_v^R$ such that the first entry in the postorder traversal $\mathcal P(\mathcal T_v^R)$ is smaller than the label of the left child of $v$. Every canonical tree is also quasicanonical; indeed, if $\mathcal T$ is a decreasing binary plane tree, then the first entry in $\mathcal P(\mathcal T)$ is less than or equal to the first entry in $\mathcal I(\mathcal T)$. 

Given a valid hook configuration $\mathcal H\in\VHC(\pi)$ for some permutation $\pi$ of size $n$, we define a quasicanonical tree $\Theta(\mathcal H)$ as follows. If $i\in[n-1]$ is a descent of $\pi$, then $(i,\pi_i)$ is the southwest endpoint of a hook $H$ in $\mathcal H$; in this case, we make $\pi_i$ a left child of $\pi_j$, where $(j,\pi_j)$ is the northeast endpoint of $H$. If $i\in[n-1]$ is not a descent of $\pi$, we make $\pi_i$ a right child of $\pi_{i+1}$. For example, if $\mathcal H$ is the valid hook configuration shown on the left in Figure~\ref{Fig19}, then $\Theta(\mathcal H)$ is the quasicanonical tree on the right. Note that quasicanonical trees are in $\mathsf{DBPT}$ by definition, so their vertices are supposed to all be black. The colors of the vertices in Figure~\ref{Fig19} are just meant to illustrate the construction and how it connects with the coloring induced by the valid hook configuration; they are not part of the definition of $\Theta(\mathcal H)$. 

\begin{figure}[h]
\begin{center} 
\includegraphics[height=5.87cm]{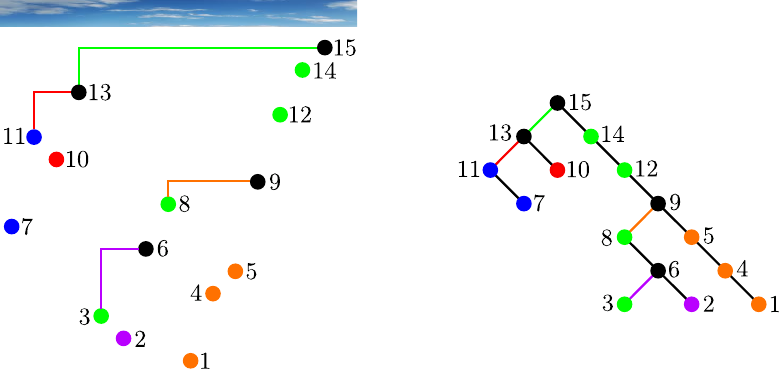}
\end{center}
\caption{On the left is a valid hook configuration $\mathcal H\in\VHC(\pi)$, where $\pi=7\,11\,10\,13\,3\,2\,6\,8\,1\,4\,5\,9\,12\,14\,15$. On the right is the quasicanonical tree $\Theta(\mathcal H)$. We have colored the hooks and points in $\mathcal H$ using the coloring described in Section~\ref{Subsec:VHCs}, and we have colored the vertices of $\Theta(\mathcal H)$ accordingly. We view $\Theta(\mathcal H)$ as a decreasing binary plane tree whose vertices are all black; the colors in the figure are just meant to illustrate the correspondence with the points in the colored diagram of $\mathcal H$. Notice also that the left edges of $\Theta(\mathcal H)$ are colored to indicate their correspondence with the hooks of $\mathcal H$.}\label{Fig19} 
\end{figure}

Given a quasicanonical tree $\mathcal T$ with $n$ vertices, we can consider the in-order traversal $\mathcal I(\mathcal T)$, which is a permutation of size $n$. Let $r_1<\cdots<r_p$ be the peaks of $\mathcal I(\mathcal T)$. We define the composition ${\bf q}^{\mathcal T}=(r_1-1,r_2-r_1-1,\ldots,r_p-r_{p-1}-1,n-r_p)$. For example, if $\mathcal T$ is the quasicanonical tree on the right in Figure~\ref{Fig19}, then $\mathcal I(\mathcal T)=11\,7\,13\,10\,15\,14\,12\,8\,3\,6\,2\,9\,5\,4\,1$. The peaks of this permutation are $r_1=3$, $r_2=5$, $r_3=10$, and $r_4=12$, so ${\bf q}^{\mathcal T}=(2,1,4,1,3)$.

\begin{theorem}\label{Thm55}
The map $\Theta$ is a bijection from the set of valid hook configurations to the set of quasicanonical trees. Furthermore, ${\bf q^{\mathcal H}}={\bf q^{\Theta(\mathcal H)}}$ for every valid hook configuration $\mathcal H$. 
\end{theorem}  

\begin{proof}
Let $\pi=\pi_1\cdots\pi_n$ be a permutation. We first argue that if $\mathcal H\in\VHC(\pi)$, then $\Theta(\mathcal H)$ is indeed quasicanonical. To see this, suppose $\pi_j$ is a vertex in $\Theta(\mathcal H)$ with a left child $\pi_i$. By construction, there is a hook in $\mathcal H$ with southwest endpoint $(i,\pi_i)$ and northeast endpoint $(j,\pi_j)$. Furthermore, the vertex $\pi_j$ has a right subtree, and the first entry in the postorder traversal of that subtree is $\pi_{i+1}$. We know that $\pi_i>\pi_{i+1}$ because $(i,\pi_i)$ is the southwest endpoint of a hook. This proves that $\Theta(\mathcal H)$ is quasicanonical. One can readily check that $\pi=\mathcal P(\Theta(\mathcal H))$. The hooks of $\mathcal H$ can be recovered from the left edges of $\Theta(\mathcal H)$, so $\Theta$ is injective. 

Let us now describe the inverse map $\Theta^{-1}$. Given a quasicanonical tree $\mathcal T$, we let $\pi=\mathcal P(\mathcal T)$ and define $\mathcal H$ to be the valid hook configuration of $\pi$ whose hooks correspond to the left edges in $\mathcal T$. More precisely, a left edge in $\mathcal T$ with endpoints $a$ and $b$ corresponds to a hook in $\mathcal H$ whose endpoints have heights $a$ and $b$. It follows from the definition of a quasicanonical tree that the southwest endpoints of the hooks of $\mathcal H$ are the descent tops of $\pi$. It is straightforward to check that $\mathcal H$ satisfies Conditions \ref{Item2} and \ref{Item3} in Definition~\ref{Def5}, so $\mathcal H$ is indeed a valid hook configuration. It follows from our construction that $\mathcal H=\Theta(\mathcal T)$, so $\Theta$ is surjective. 

We now prove that ${\bf q}^{\mathcal H}={\bf q}^{\Theta(\mathcal H)}$. Suppose $\mathcal H=(H_1,\ldots,H_k)\in\VHC(\pi)$ for some permutation $\pi=\pi_1\cdots\pi_n$. Let $\sigma=\sigma_1\cdots\sigma_n=\mathcal I(\Theta(\mathcal H))$, and let $r_1<\cdots<r_p$ be the peaks of $\sigma$ so that ${\bf q}^{\Theta(\mathcal H)}=(r_1-1,r_2-r_1-1,\ldots,r_p-r_{p-1}-1,n-r_p)$. The entries $\sigma_{r_1},\ldots,\sigma_{r_p}$ are precisely the labels of the vertices of $\Theta(\mathcal H)$ that have $2$ children (equivalently, that have left children). Moreover, we have $p=k$, and $\sigma_{r_\ell}$ is the height of the northeast endpoint of $H_\ell$ for all $1\leq \ell\leq k$. The heights of the points in the plot of $\pi$ that lie below $H_\ell$ are the labels of the vertices of $\Theta(\mathcal H)$ that lie in the right subtree of the vertex with label $\sigma_{r_\ell}$. It follows that a point $(i,\pi_i)$ receives the same color as $H_\ell$ in the coloring induced by $\mathcal H$ (i.e., $H_\ell$ is the lowest hook lying above $(i,\pi_i)$) if and only if $\pi_i=\sigma_j$ for some $r_\ell<j<r_{\ell+1}$ (with the conventions $r_0=0$ and $r_{p+1}=n+1$). Thus, if ${\bf q}^{\mathcal H}=(q_0,\ldots,q_k)$, then $q_\ell=r_{\ell+1}-r_\ell-1$.   
\end{proof}

Theorem~\ref{Thm55} allows us to rewrite the Refined Tree Fertility Formula (and its corollaries) in terms of quasicanonical trees. We can do the same with the VHC Cumulant Formula, producing the following new combinatorial formula for converting from free to classical cumulants. Let $\QCan_n$ denote the set of quasicanonical trees on $[n]$.

\begin{corollary}\label{Cor:Fertilitopes2}
If $(\kappa_n)_{n\geq 1}$ is a sequence of free cumulants, then the corresponding classical cumulants are given by \[-c_n=\sum_{\mathcal T\in\QCan_{n-1}}(-\kappa_{\bullet+1})_{\bf q^{\mathcal T}}.\]  
\end{corollary}

\begin{proof}
By the VHC Cumulant Formula (Theorem~\ref{VHCCF}), Theorem~\ref{Thm55}, and Remark~\ref{Rem6}, we have \[-c_n=\sum_{\pi\in S_{n-1}}\sum_{{\bf q}\in\mathcal V(\pi)}(-\kappa_{\bullet+1})_{\bf q}=\sum_{\mathcal H\in\VHC(S_{n-1})}(-\kappa_{\bullet+1})_{\bf q^{\mathcal H}}=\sum_{\mathcal T\in\QCan_{n-1}}(-\kappa_{\bullet+1})_{{\bf q}^{\mathcal T}}.\qedhere\]
\end{proof}

Let us now turn back to the canonical trees defined by Bousquet-M\'elou. Her main motivation for defining these trees was to understand sorted permutations. A permutation is sorted if and only if it has a valid hook configuration, and this occurs if and only if it has a canonical hook configuration (as mentioned after the definition of canonical hook configurations in Section~\ref{Sec:Main}). It follows from the identity $s=\mathcal P\circ\mathcal I^{-1}$ in \eqref{sPcircI} that a permutation $\pi$ is sorted if and only if there is a decreasing binary plane tree with postorder traversal $\pi$. Bousquet-M\'elou proved that if $\pi$ is sorted, then there is a unique canonical tree $\mathcal T$ with postorder traversal $\pi$. Moreover, she showed
that $\mathcal I(\mathcal T)$ is the unique element of $s^{-1}(\pi)$ with the maximum number of inversions. 

\begin{theorem}
Let $\pi\in S_n$ be a sorted permutation, and let $\mathcal H^*$ be its canonical hook configuration. Then $\Theta(\mathcal H^*)$ is the unique canonical tree with postorder traversal $\pi$. 
\end{theorem}

\begin{proof}
We have seen that $\mathcal P(\Theta(\mathcal H^*))=\pi$, so it suffices to prove that $\Theta(\mathcal H^*)$ is canonical. Let $v$ be a vertex of $\Theta(\mathcal H^*)$ that has a left child with label $a$. Let $b$ be the label of $v$. Because $\Theta(\mathcal H^*)$ is quasicanonical, $v$ has a nonempty right subtree $\mathcal T_v^R$. There is a hook $H$ in $\mathcal H^*$ whose southwest endpoint has height $a$ and whose northeast endpoint has height $b$. Let $c$ be the first entry in $\mathcal I(\mathcal T_v^R)$, and note that $c<b$ because $\Theta(\mathcal H^*)$ is a decreasing binary plane tree. Suppose, by way of contradiction, that $c>a$. We saw in the proof of Theorem~\ref{Thm55} that $c$ is the height of a point in the plot of $\pi$ that receives the same color as $H$ in the coloring induced by $\mathcal H^*$. This means that there is a hook $H'$ (not in $\mathcal H^*$) whose southwest endpoint has height $a$ and whose northeast endpoint has height $c$; it also means that we can obtain a new valid hook configuration of $\pi$ from $\mathcal H^*$ by replacing $H$ with $H'$. However, this contradicts the fact that the canonical hook configuration $\mathcal H^*$ is constructed so that all of the northeast endpoints of the hooks are as low as possible. We deduce that $c<a$; as $v$ was arbitrary, this proves that $\Theta(\mathcal H^*)$ is canonical.  
\end{proof}

We now return to fertilitopes and binary nestohedra. A permutation is sorted if and only if its fertilitope is nonempty. Now suppose $\pi\in S_n$ is a sorted permutation with $k$ descents, and let $\mathcal H^*$ be its canonical hook configuration. Theorem~\ref{Thm:Main2} tells us that the fertilitope $\Fer_\pi$ is an integral binary nestohedron. We will show how the canonical tree $\Theta(\mathcal H^*)$ immediately yields a binary building set $\mathcal B^\pi$ and a tuple of positive integers ${\bf y}^\pi=(y_I^\pi)_{I\in\mathcal B^\pi}$ such that $\Fer_\pi=\Nest({\bf y}^\pi)$. 

Let $T^\pi=\skel(\Theta(\mathcal H^*))$ be the tree obtained by removing the labels from $\Theta(\mathcal H^*)$. Now assign each leaf of $T^\pi$ the label $1$, and assign each internal vertex the label $0$. Suppose there is a vertex $v$ of $T^\pi$ that has no left child and has a right child $u$, where $u$ is either a leaf or a vertex with $2$ children. Contract the edge connecting $u$ and $v$ into a single vertex, and identify the new vertex with the original vertex $u$. Also, increase the label of $u$ by $1$. Continue repeating this process of contracting edges until the resulting tree is full (i.e., every vertex is either a leaf or has $2$ children). Call this resulting tree $\widehat T^\pi$. It is straightforward to check that $\widehat T^\pi$ does not depend on the order in which we contracted the edges.

The tree $T^\pi$ has $k$ left edges, and none of these edges were contracted when we formed $\widehat T^\pi$. Therefore, $\widehat T^\pi$ has $k$ left edges, $k$ right edges, and $k+1$ leaves. Identify the leaves of $\widehat T^\pi$ with the singleton sets $\{1\},\ldots,\{k+1\}$ from left to right, and identify each internal vertex with the union of the leaves lying below it. For each vertex $I$, let $y_I^\pi$ be the label assigned to the vertex $I$. Finally, let $\mathcal B^\pi$ be the set of vertices $I$ of $\widehat T^\pi$ such that $y_I^\pi>0$, and put ${\bf y}^\pi=(y_I^\pi)_{I\in\mathcal B^\pi}$.  

\begin{figure}[h]
\begin{center} 
\includegraphics[height=10.48cm]{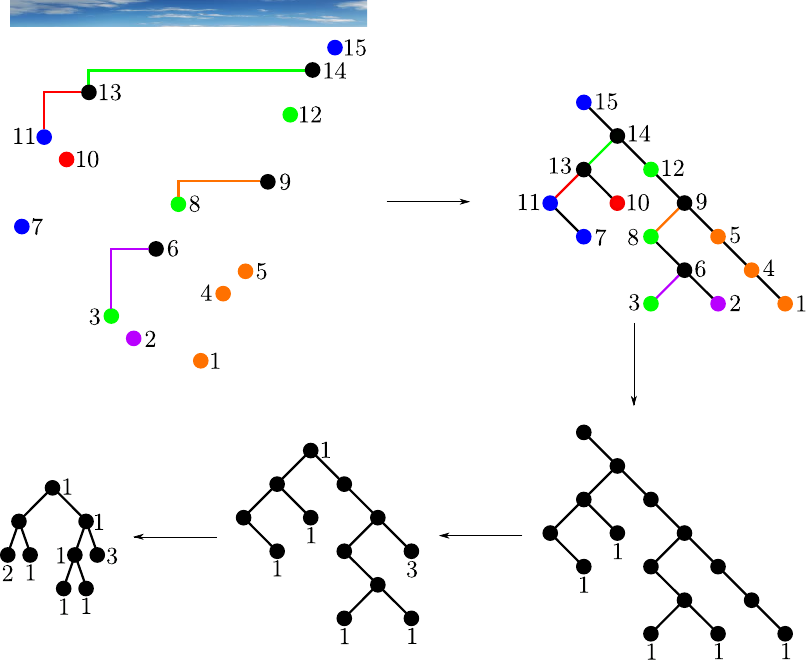}
\end{center}
\caption{The construction of $\widehat T^\pi$ from the canonical hook configuration $\mathcal H^*$ of $\pi$.}\label{FigFertilitopes1} 
\end{figure}

\begin{example}
Let $\pi=7\,11\,10\,13\,3\,2\,6\,8\,1\,4\,5\,9\,12\,14\,15$. The canonical hook configuration $\mathcal H^*$ of $\pi$ is shown (with its induced coloring) in the upper left of Figure~\ref{FigFertilitopes1}. The canonical tree $\Theta(\mathcal H^*)$ appears in the upper right of Figure~\ref{FigFertilitopes1}. In the bottom right of the same figure, we have the tree $T^\pi=\skel(\Theta(\mathcal H^*))$, where each leaf has been assigned the label $1$. The tree in the bottom middle is the result of applying three edge contractions to $T^\pi$. After applying three more edge contractions, we obtain the tree $\widehat T^\pi$ shown in the bottom left. All vertices that do not appear with labels are assumed to have label $0$.  

The labels of the vertices of $\widehat T^\pi$ translate to the numbers $y_{\{1\}}^\pi=2$, $y_{\{2\}}^\pi=1$, $y_{\{3\}}^\pi=1$, $y_{\{4\}}^\pi=1$, $y_{\{5\}}^\pi=3$, $y_{\{1,2\}}^\pi=0$, $y_{\{3,4\}}^\pi=1$, $y_{\{3,4,5\}}^\pi=1$, and $y_{\{1,2,3,4,5\}}^\pi=1$. Consequently, \[\mathcal B^\pi=\{\{1\},\{2\},\{3\},\{4\},\{5\},\{3,4\},\{3,4,5\},\{1,2,3,4,5\}\}.\qedhere\] 
\end{example}

\begin{theorem}\label{Thm:Fertilitopes3}
Let $\pi$ be a sorted permutation with $k$ descents, and let $\mathcal B^\pi$ and ${\bf y}^\pi=(y_I^\pi)_{I\in\mathcal B^\pi}$ be as defined above. Then $\Fer_\pi=\Nest({\bf y}^\pi)$. 
\end{theorem}

\begin{proof}
If $\pi=\id_n$, then the tree $T^\pi=\skel(\Theta(\mathcal H^*))$ has $n$ vertices, $n-1$ right edges, and no left edges. After contracting the $n-1$ edges, we are left with the tree $\widehat T^{\id_n}$, which has a single vertex labeled $n$. Then $\mathcal B^\pi=\{\{1\}\}$ and $y_{\{1\}}^\pi=n$, so $\Nest({\bf y}^\pi)=\{(n)\}=\Fer_{\id_n}$. We may now assume $\pi\neq\id_n$ (hence, $n\geq 2$) and proceed by induction on $n$. Being sorted, the permutation $\pi$ ends with the entry $n$, so we can write $\pi=\pi'\oplus 1$ for some $\pi'=S_{n-1}$. 

Suppose first that $\mathcal V(\pi')\neq\emptyset$. Proposition~\ref{Prop1} tells us that $\Fer_{\pi}=\Fer_{\pi'}+\Delta_{[k+1]}$. Let $(\mathcal H')^*$ be the canonical hook configuration of $\pi'$. The hooks in $(\mathcal H')^*$ are the same as those in $\mathcal H^*$; in particular, the point $(n,n)$ in the plot of $\pi$ is not the northeast endpoint of a hook in $\mathcal H^*$. This implies that the vertex with label $n$ in the canonical tree $\Theta(\mathcal H^*)$ (i.e., the root) has no left child. The right subtree of this root vertex is $\Theta((\mathcal H')^*)$. Let $T^{\pi'}=\skel(\Theta((\mathcal H')^*))$, and let $\widehat T^{\pi'}$ be the full binary plane tree obtained from $T'$ via the edge-contraction process described above. When we contract the edge of $T^\pi$ whose upper vertex is the root, we obtain the tree $T^{\pi'}$, except with the root labeled $1$ instead of $0$. Therefore, $\widehat T^\pi$ is obtained from $\widehat T^{\pi'}$ by increasing the label of the root vertex (i.e., the vertex identified with the set $[k+1]$) by $1$. We have ${\bf y}^{\pi'}=(y_I^{\pi'})_{I\in\mathcal B^{\pi'}}$, where $y_I^{\pi'}$ is the label of the vertex $I$ in $\widehat T^{\pi'}$ and $\mathcal B^{\pi'}=\{I:y_I^{\pi'}>0\}$. Then $y_{[k+1]}^{\pi}=y_{[k+1]}^{\pi'}+1$ and $y_I^{\pi}=y_I^{\pi'}$ for all vertices $I\neq[k+1]$. It follows by induction on $n$ that $\Fer_{\pi'}=\Nest({\bf y}^{\pi'})$, so \[\Fer_{\pi}=\Fer_{\pi'}+\Delta_{[k+1]}=\Nest({\bf y}^{\pi'})+\Delta_{[k+1]}=\Nest({\bf y}^{\pi}).\]

Next, assume $\mathcal V(\pi')=\emptyset$. We saw in the proof of Theorems~\ref{Thm:Main1} and \ref{Thm:Main2} that there is a hook $H_d^*$ in the canonical hook configuration $\mathcal H^*$ of $\pi$ whose northeast endpoint is $(n,n)$. In addition, we saw that \[\Fer_\pi=\Fer_{\pi_U^{H_d^*}}\times\Fer_{\pi_S^{H_d^*}}.\] Let $\mathcal H_U^*$ and $\mathcal H_S^*$ be the canonical hook configurations of $\pi_U^{H_d^*}$ and $\pi_S^{H_d^*}$, respectively. The hooks of $\mathcal H_U^*$ (respectively, $\mathcal H_S^*$) are precisely the hooks of $\mathcal H^*$ that are hooks of $\pi_U^{H_d^*}$ (respectively, $\pi_S^{H_d^*}$). Therefore, the left (respectively, right) subtree of the root of $\Theta(\mathcal H^*)$ is $\Theta(\mathcal H_U^*)$ (respectively, $\Theta(\mathcal H_S^*)$). Upon inspecting the edge-contraction process, we find that $\widehat T^\pi$ consists of a root vertex with label $0$ whose left and right subtrees are $\widehat T^{\pi_U^{H_d^*}}$ and $\widehat T^{\pi_S^{H_d^*}}$, respectively, except that the vertices of the right subtree are identified with different subsets. More precisely, the tree $\widehat T^{\pi_U^{H_d^*}}$ has $d$ leaves, so each vertex $I=\{a_1,\ldots,a_r\}$ of $\widehat T^{\pi_S^{H_d^*}}$ is identified with the set $I+d=\{a_1+d,\ldots,a_r+d\}$ in $\widehat T^\pi$. Let ${\bf y}_U$ and ${\bf y}_S$ be the tuples of nonzero labels of vertices of $\widehat T^{\pi_U^{H_d^*}}$ and $\widehat T^{\pi_S^{H_d^*}}$, respectively. Let ${\bf y}_S^+$ be the set of nonzero labels of the vertices in the right subtree of the root of $\widehat T^\pi$. Then ${\bf y}_S$ and ${\bf y}_S^+$ are essentially the same tuple of numbers: each number $y_I$ labeling the vertex $I$ in the first tuple is equal to the number $y_{I+d}$ labeling the corresponding vertex $I+d$ in the second tuple. We can now use induction on $n$ to deduce that 
\[\Fer_\pi=\Fer_{\pi_U^{H_d^*}}\times\Fer_{\pi_S^{H_d^*}}=\Nest({\bf y}_U)\times \Nest({\bf y}_S)=\Nest({\bf y}).\qedhere\] 
\end{proof}

\begin{remark}
By combining Theorem~\ref{Thm:Fertilitopes3} with Postnikov's Theorem~\ref{Thm:PostnikovFN}, we can read off the dimension and the face numbers of $\Fer_\pi$ from the canonical hook configuration of $\pi$. The number of facets is particularly easy to compute because it is simply $|\mathcal B^\pi\setminus\mathcal B_{\max}^\pi|$ when $\dim(\Fer_\pi)>0$. 
\end{remark}

\section{Fertility Numbers}\label{Sec:FertilityNumbers}

For many years, the stack-sorting map was thought to be extremely complicated and devoid of structure. However, this turns out to be far from the truth. Much of the structure underlying the map $s$ comes from the Fertility Formula. For example, there are no permutations $\pi$ such that $|s^{-1}(\pi)|=3$; this would certainly not be the case if $s$ truly behaved in a chaotic fashion. 

We say a nonnegative integer $f$ is a \dfn{fertility number} if it is equal to the fertility of some permutation. An \dfn{infertility number} is a nonnegative integer that is not a fertility number. According to the Fertility Formula (Corollary~\ref{FF}), Theorem~\ref{Thm:Main1}, and Theorem~\ref{Thm:Main2}, a positive integer $f$ is a fertility number if and only if it is of the form 
\[\sum_{{\bf q}\in \Nest({\bf y})\cap\,\mathbb Z^{k+1}}C_{\bf q}\] for some integral binary nestohedron $\Nest({\bf y})$. Our goal in this section is to use this reinterpretation of fertility numbers to gain new information about them. 

The basic properties of fertility numbers that were proven in \cite{DefantFertility} are listed in Section~\ref{Sec:Intro}. In particular, it was shown that the set of fertility numbers is a multiplicative monoid with lower asymptotic density at least $0.7618$ that contains all nonnegative integers that are not congruent to $3$ modulo $4$. In \cite{DefantFertility}, the author asked if the set of fertility numbers has a density and, if so, what this density is. We will answer this question by showing that the set of fertility numbers has density $1$. In \cite{DefantFertility}, it was also proven that $3,7,11,15,19,23$ are all infertility numbers; we will use Theorem~\ref{Thm:Main1} to extend this list by determining all infertility numbers that are at most $126$. 

\begin{lemma}\label{LemFertilitopes2}
For each nonnegative integer $f$, the following are equivalent: 
\begin{enumerate}
\item There exist an integer $n\geq 3$ and a permutation $\pi\in S_n$ satisfying $\des(\pi)=\frac{n-3}{2}$ and $|s^{-1}(\pi)|=f$.
\item The integer $f$ is divisible by $4$ or is of the form $a(-2a+4b+3)$ for some integers $1\leq a\leq b$. 
\end{enumerate} 
\end{lemma}

\begin{proof}
Both statements holds when $f=0$, so we may assume that $f\geq 1$. 

Suppose there exist $n\geq 3$ and $\pi\in S_n$ satisfying $\des(\pi)=k=\frac{n-3}{2}$ and $|s^{-1}(\pi)|=f$. Since $f\geq 1$, the fertilitope $\Fer_\pi$ is nonempty. By Corollary~\ref{Cor1} and Theorem~\ref{Thm:Main2}, there is a binary building set $\mathcal B$ and a tuple ${\bf y}=(y_I)_{I\in\mathcal B}\in\mathbb Z_{>0}^{\mathcal B}$ such that $\Fer_\pi=\Nest({\bf y})$.
Every valid composition of $\pi$ is a composition of $n-k$ into $k+1$ parts, so $\sum_{I\in\mathcal B}y_I=n-k$. Furthermore, $y_{\{i\}}\geq 1$ for every $i\in[k+1]$. Because $n-k=k+3$, we conclude that there are (not necessarily distinct) sets $J,J'\in\mathcal B$ such that $\Fer_\pi=\sum_{i=1}^{k+1}\Delta_{\{i\}}+\Delta_J+\Delta_{J'}$. It now follows from Theorems~\ref{Thm:PostLattice} and \ref{Thm:Main1} that 
\begin{equation}\label{EqFertilitopes2}
\mathcal V(\pi)=\Fer_\pi\cap\,\mathbb Z^{k+1}=\{(1,1,\ldots,1)+e_j+e_{j'}:j\in J,\, j'\in J'\}.
\end{equation} 

If $J\cap J'=\emptyset$, then each composition in $\mathcal V(\pi)$ contains two parts equal to $2$ and $k-1$ parts equal to $1$. In this case, we have $C_{\bf q}=4$ for every ${\bf q}\in\mathcal V(\pi)$, so the Fertility Formula (Corollary~\ref{FF}) tells us that $f=|s^{-1}(\pi)|=4|\mathcal V(\pi)|$ is divisible by $4$.

Now suppose $J\cap J'\neq\emptyset$. Because $\mathcal B$ is a binary building set, we either have $J\subseteq J'$ or $J'\subseteq J$; without loss of generality, assume $J\subseteq J'$. Let $a=|J|$ and $b=|J'|$, and note that $1\leq a\leq b$. It follows from \eqref{EqFertilitopes2} that every composition in $\mathcal V(\pi)$ either has one part equal to $3$ and $k$ parts equal to $1$ or two parts equal to $2$ and $k-1$ parts equal to $1$. The number of compositions in $\mathcal V(\pi)$ of the first kind is $a$. The compositions in $\mathcal V(\pi)$ of the second kind are those of the form $(1,1,\ldots,1)+e_j+e_{j'}$ such that $j\neq j'$ and either $j,j'\in J$ or $j\in J$ and $j'\in J'\setminus J$. The number of such compositions is $\binom{a}{2}+a(b-a)$. If ${\bf q}$ is a composition of the first kind, then $C_{\bf q}=5$; if ${\bf q}$ is of the second kind, then $C_{\bf q}=4$. Therefore, it follows from the Fertility Formula (Corollary~\ref{FF}) that $f=|s^{-1}(\pi)|=5a+4\left(\binom{a}{2}+a(b-a)\right)=a(-2a+4b+3)$. This proves one direction of the lemma. 

To prove the converse, consider a positive integer $k$ and sets $J,J'\subseteq[k+1]$ (to be specified later) such that the collection $\mathcal B=\{\{1\},\ldots,\{k+1\},J,J'\}$ is a binary building set. Form the list of sets $\{1\},\{2\},\ldots,\{k+1\},J,J'$, and let $y_I$ denote the number of times that $I$ appears in the list. Let ${\bf y}=(y_I)_{I\in\mathcal B}$. It follows from Theorem~\ref{Thm:Main2} that there is a permutation $\pi$ with $k$ descents such that $\Fer_\pi=\Nest({\bf y})$. The size of $\pi$ is $k+\sum_{I\in\mathcal B}y_I=2k+3$. We have $\mathcal V(\pi)=\{(1,1,\ldots,1)+e_j+e_{j'}:j\in J,\,j'\in J'\}$.

If $f$ is divisible by $4$, then let $k=f/4$, $J=[k]$, and $J'=\{k+1\}$. 
In this case, $|\mathcal V(\pi)|=k$, and $C_{\bf q}=4$ for all ${\bf q}\in\mathcal V(\pi)$. By the Fertility Formula, $|s^{-1}(\pi)|=4k=f$. 

Finally, suppose $f=a(-2a+4b+3)$ for some positive integers $a\leq b$. Let $k=b-1$, $J=[a]$, and $J'=[b]$. Then $\mathcal V(\pi)$ contains $a$ compositions of the form $(1,1,\ldots,1)+2e_j$ for $j\in J$. It also contains $\binom{a}{2}+a(b-a)$ compositions of the form $(1,1,\ldots,1)+e_j+e_{j'}$ for $j\in J$, $j'\in J'$, and $j\neq j'$. The compositions ${\bf q}$ of the first kind satisfy $C_{\bf q}=5$, while the compositions ${\bf q}$ of the second kind satisfy $C_{\bf q}=4$. By the Fertility Formula, $|s^{-1}(\pi)|=5a+4\left(\binom{a}{2}+a(b-a)\right)=a(-2a+4b+3)=f$. 
\end{proof}

\begin{theorem}\label{Thm:Fertilitopes2}
The set of fertility numbers has density $1$ in the set of nonnegative integers. 
\end{theorem}

\begin{proof}
For each prime number $p\equiv 3\pmod 4$, let $A_p=\{p(-2p+4b+3):b\in\mathbb Z, b\geq p\}$. Every number in $A_p$ is congruent to $3$ modulo $4$. The density of $A_p$ is $1/(4p)$. Let $A=\bigcup_{p\equiv 3\pmod 4}A_p$, where the union is over all primes $p$ that are congruent to $3$ modulo $4$. Let $B$ be the set of positive integers that are congruent to $3$ modulo $4$ and are not in $A$. By the Chinese Remainder Theorem, the density of $B$ is \[\frac{1}{4}\prod_{p\equiv 3\!\!\!\!\!\pmod 4}\left(1-\frac{1}{p}\right);\] basic analytic number theory tells us that the value of this infinite product is $0$. It follows that $A$ has density $1/4$. According to Lemma~\ref{LemFertilitopes2}, every element of $A$ is a fertility number. In \cite{DefantFertility}, the author proved that every nonnegative integer that is not congruent to $3$ modulo $4$ is a fertility number. This completes the proof.
\end{proof}

\begin{theorem}\label{Thm:Fertilitopes1}
A nonnegative integer $f\leq 126$ is an infertility number if and only if $f\equiv 3\pmod 4$, $f\not\in\{95,119\}$, and either $f<27$ or $f\not\equiv 3\pmod{12}$. 
\end{theorem}

\begin{proof}
As we have already remarked, it was proven in \cite{DefantFertility} that every infertility number is congruent to $3$ modulo $4$. The valid compositions of $1243567$ are $(3,3)$, $(4,2)$, and $(5,1)$, so $C_{(3,3)}+C_{(4,2)}+C_{(5,1)}=95$ is a fertility number. By setting $a=b=7$ in Lemma~\ref{LemFertilitopes2}, we find that $119$ is a fertility number. By setting $a=3$ in Lemma~\ref{LemFertilitopes2}, we find that every integer $m\geq 27$ such that $m\equiv 3\pmod{12}$ is a fertility number. This proves one direction of the theorem. 

To prove the other direction, suppose $f\equiv 3\pmod 4$, $f\not\in\{95,119\}$, and either $f<27$ or $f\not\equiv 3\pmod{12}$; combined with the hypothesis $f\leq 126$, these conditions imply that $f\leq 115$. We will demonstrate that $f$ is an infertility number. Suppose, by way of contradiction, that there is a permutation $\pi$ of size $n$ such that $|s^{-1}(\pi)|=f$. Since permutations with the same relative order have the same fertility, we may assume $\pi\in S_n$. Let $k=\des(\pi)$. The Fertility Formula (Corollary~\ref{FF}) tells us that $f=\sum_{{\bf q}\in\mathcal V(\pi)}C_{\bf q}$, and we know that every element of $\mathcal V(\pi)$ is a composition of $n-k$ into $k+1$ parts. Hence, $n-k\geq k+1$. The only odd Catalan numbers $C_r$ with $r\geq 1$ and $C_r\leq 115$ are $C_1=1$ and $C_3=5$. Therefore, we deduce from the fact that $f$ is odd that there exists a composition ${\bf q^{\text{odd}}}\in\mathcal V(\pi)$ whose parts are all equal to either $1$ or $3$. Because ${\bf q}^{\text{odd}}$ is a composition of $n-k$ into $k+1$ parts, we must have $n-k\equiv k+1\pmod 2$. Also, ${\bf q}^{\text{odd}}$ has $(n-2k-1)/2$ parts equal to $3$, so $115\geq f\geq C_{{\bf q}^{\text{odd}}}=C_3^{(n-2k-1)/2}=5^{(n-2k-1)/2}$. Together, these observations tell us that $n-k\in\{k+1,k+3,k+5\}$. 

If $n-k=k+1$, then we must have $\mathcal V(\pi)=\{(1,1,\ldots,1)\}$. However, this implies that $f=C_{(1,1,\ldots,1)}=1$, which contradicts the assumption that $f\equiv 3\pmod 4$. 

If $n-k=k+3$, then it follows from Lemma~\ref{LemFertilitopes2} that $f=a(-2a+4b+3)$ for some positive integers $a\leq b$. In this case, the hypothesis $f\equiv 3\pmod 4$ forces $a\equiv 3\pmod 4$, and the fact that $f\leq 115$ forces $a<7$. However, this is impossible because it implies that $a=3$ and, therefore, $27\leq f=3(4b-3)\equiv 3\pmod{12}$.  

Finally, suppose $n-k=k+5$. We have seen, as a consequence of Theorems~\ref{Thm:Main1} and \ref{Thm:Main2}, that $\mathcal V(\pi)$ is a discrete polymatroid contained in the set $\Comp_{k+1}(k+5)\coloneqq\{(x_0,\ldots,x_k)\in\mathbb Z_{>0}^{k+1}:x_0+\cdots+x_k=k+5\}$. Therefore, in order to obtain our desired contradiction in this case, it suffices to prove that $\sum_{{\bf q}\in V}C_{{\bf q}}\neq f$ for \emph{every} discrete polymatroid $V\subseteq \Comp_{k+1}(k+5)$. Suppose instead that there is a discrete polymatroid $V\subseteq \Comp_{k+1}(k+5)$ such that $\sum_{{\bf q}\in V}C_{{\bf q}}=f$. As above, we can use the fact that $f$ is odd to deduce that there is a vector in $V$ whose parts are all equal to $1$ or $3$. By permuting coordinates if necessary, we may assume this vector is $(3,3,1,1,\ldots,1)$. It will be helpful to keep in mind the easily-verified fact that $C_{\bf q}\geq 16$ for every ${\bf q}\in \Comp_{k+1}(k+5)$. We consider several cases. 

\medskip
\noindent {\bf Case 1.} Suppose $V$ contains either $(5,1,1,\ldots,1)$ or $(1,5,1,1,\ldots,1)$. Without loss of generality, we may assume $V$ contains $(5,1,1,\ldots,1)$. Because $V$ is a discrete polymatroid that contains $(3,3,1,1,\ldots,1)$ and $(5,1,1,\ldots,1)$, it must also contain $(4,2,1,1,\ldots,1)$. Since \[C_{(3,3,1,1,\ldots,1)}+C_{(5,1,1,\ldots,1)}+C_{(4,2,1,1,\ldots,1)}=95\neq f,\] there must be some other vector ${\bf q}'=(q_0',\ldots,q_k')\in V$. Using the fact that $f\leq 115$, we readily check that $|V|=4$, that $q_i'\leq 3$ for all $i\in\{0,\ldots,k\}$, and that there is an index $j\geq 2$ with $q_j'\geq 2$. Because $V$ is a discrete polymatroid that contains both ${\bf q}'$ and $(5,1,1,\ldots,1)$, it must contain a vector $(q_0'',\ldots,q_k'')$ with $q_0''=4$ and $q_j''=2$. This forces $|V|\geq 5$, which is a contradiction. 

\medskip
\noindent {\bf Case 2.} Suppose $V$ contains a vector ${\bf q}'=(q_0',\ldots,q_k')$ with $q_j'=5$ for some $j\geq 2$. For $\ell\in\{0,1\}$, let ${\bf q}^{(\ell)}=(q_0^{(\ell)},\ldots,q_k^{(\ell)})$ be the vector satisfying $q_\ell^{(\ell)}=2$, $q_j^{(\ell)}=4$, and $q_i^{(\ell)}=1$ for all $i\in\{0,\ldots,k\}\setminus\{\ell,j\}$. Because $V$ is a discrete polymatroid that contains both $(3,3,1,1,\ldots,1)$ and ${\bf q}'$, it must contain both ${\bf q}^{(0)}$ and ${\bf q}^{(1)}$. This is a contradiction because it implies that \[f=\sum_{{\bf q}\in V}C_{\bf q}\geq C_{(3,3,1,1,\ldots,1)}+C_{{\bf q}'}+C_{{\bf q}^{(0)}}+C_{{\bf q}^{(1)}}=25+42+28+28=123.\]

\medskip
\noindent {\bf Case 3.} Suppose that none of the vectors in $V$ has a coordinate equal to $5$. The \dfn{type} of a composition is the integer partition obtained by rearranging its parts into nonincreasing order. The possible types of the compositions in $V$ are $(4,2,1,1,\ldots,1)$, $(3,2,2,1,1,\ldots,1)$, $(2,2,2,2,1,1,\ldots,1)$, and $(3,3,1,1,\ldots,1)$. If ${\bf q}$ is a composition of one of the first three types, then $C_{\bf q}\equiv 0\pmod 4$. Since $f\equiv 3\pmod 4$ and $f\leq 115$, there must be exactly three compositions of type $(3,3,1,1,\ldots,1)$ in $V$. Let ${\bf q}^{(1)}$, ${\bf q}^{(2)}$, ${\bf q}^{(3)}$ be these compositions. For $1\leq i<j\leq 3$, let ${\bf q}^{(i,j)}=({\bf q}^{(i)}+{\bf q}^{(j)})/2$; note that ${\bf q}^{(i,j)}\in \Comp_{k+1}(k+5)$. Being a discrete polymatroid, $V$ is the set of lattice points in a polytope (in fact, a generalized permutohedron), so the vectors ${\bf q}^{(1,2)}$, ${\bf q}^{(1,3)}$, ${\bf q}^{(2,3)}$ must all belong to $V$. This is our desired contradiction because it implies that \[f=\sum_{{\bf q}\in V}C_{\bf q}\geq \sum_{i=1}^3C_{{\bf q}^{(i)}}+\sum_{1\leq i<j\leq 3}C_{{\bf q}^{(i,j)}}=\sum_{i=1}^325+\sum_{1\leq i<j\leq 3}C_{{\bf q}^{(i,j)}}\geq 75+\sum_{1\leq i<j\leq 3}16=123.\qedhere\]
\end{proof}

\begin{remark}
The methods used in the proof of Theorem~\ref{Thm:Fertilitopes1} could be extended to determine even more infertility numbers, but this would require much further tedious casework. 
\end{remark}

In \cite{DefantFertility}, the author conjectured that there are infinitely many infertility numbers. However, in light of Theorem~\ref{Thm:Fertilitopes2}, he is now more doubtful of this conjecture. Hence, we have rephrased this conjecture as a question. 
 
\begin{question}
Are the infinitely many infertility numbers?
\end{question}

\section{Further Questions}\label{Sec:Descent}

\subsection{Real-Rooted Polynomials from Polytopes}
We say a polynomial $a_0+a_1x+\cdots+a_mx^m\in\mathbb R_{\geq 0}[x]$ is \dfn{real-rooted} if all of its roots are real, and we say it is \dfn{log-concave} if $a_{j-1}a_{j+1}\leq a_j^2$ for all $j\in[m-1]$. It is well known that every real-rooted polynomial in $\mathbb R_{\geq 0}[x]$ is log-concave. Recall Conjectures~\ref{Conj2} and \ref{Conj:Fertilitopes1} from Section~\ref{Sec:Intro}. Despite the fact that these conjectures appear, at first glance, to be very far removed from each other, we saw in Corollary~\ref{Cor:Fertilitopes1} that they are equivalent.

One of the advantages of Conjecture~\ref{Conj:Fertilitopes1} over Conjecture~\ref{Conj2} is that it appears to fit naturally into a broader context; in other words, it is likely that the ``correct'' form of this conjecture (assuming the conjecture itself is correct) is more general. For one thing, the polynomials in Conjecture~\ref{Conj:Fertilitopes1} are given by sums over the lattice points in very specific nestohedra, and nestohedra are special cases of generalized permutohedra. 

\begin{question}\label{Quest:Fertilitopes1}
Is it true that $\displaystyle\sum_{{\bf q}\in P\cap\,\mathbb Z^{k+1}}N_{{\bf q}}(x)$ is real-rooted whenever $P\subseteq \mathbb R_{>0}^{k+1}$ is an integral nestohedron? Is this true whenever $P\subseteq \mathbb R_{>0}^{k+1}$ is an integral generalized permutohedron?
\end{question}

Note that the second part of Question~\ref{Quest:Fertilitopes1} could be equivalently rephrased in terms of discrete polymatroids since discrete polymatroids are essentially the same as sets of lattice points of generalized permutohedra. We can also consider lattice points in more general polytopes. It is natural to require these polytopes to be integral and to lie in the set $\{(x_0,\ldots,x_k)\in\mathbb R_{>0}^{k+1}:x_0+\cdots+x_k=\ell\}$ for some $\ell$.  

\begin{question}
For nonnegative integers $k$ and $\ell$, what can be said about the set of integral polytopes $P\subseteq\{(x_0,\ldots,x_k)\in\mathbb R_{>0}^{k+1}:x_0+\cdots+x_k=\ell\}$ such that $\displaystyle\sum_{{\bf q}\in P\cap\,\mathbb Z^{k+1}}N_{{\bf q}}(x)$ is real-rooted? Can we exhibit an integral polytope $P\subseteq\{(x_0,\ldots,x_k)\in\mathbb R_{>0}^{k+1}:x_0+\cdots+x_k=\ell\}$ such that this polynomial is not real-rooted? 
\end{question}

In another direction, we could ask if Conjecture~\ref{Conj:Fertilitopes1} holds when the sequence of Narayana polynomials is replaced by more general sequences of real-rooted polynomials. What happens if we replace Narayana polynomials with, say, Eulerian polynomials? 

Because Conjecture~\ref{Conj:Fertilitopes1} involves polynomials that are defined as sums over discrete polymatroids (which are M-convex sets), it appears to be related, at least superficially, to the theory of Lorentzian polynomials introduced recently by Br\"and\'en and Huh \cite{BrandenHuh}; it would be very interesting if this connection was more than superficial. 

In all of these questions, one could just as well replace ``real-rooted'' with ``log-concave'' and still obtain valid questions worth pondering.  

\subsection{Extremal Hook Configurations}
There are several interesting enumerative properties of valid hook configurations. These objects were counted in \cite{DefantEngenMiller, DefantTroupes} using free probability theory. The articles \cite{Ilani, DefantCatalan, DefantMotzkin, Hanna, Maya} analyze valid hook configurations whose underlying permutations avoid certain patterns. The articles \cite{DefantCatalan, DefantEngenMiller, DefantTroupes, Hanna, Singhal} study \emph{uniquely sorted permutations}, which are essentially valid hook configurations with the maximum possible number of hooks. Furthermore, Sankar introduced the notion of a \emph{reduced} valid hook configuration, which was subsequently studied further by Axelrod-Freed \cite{Ilani}. 

According to Remark~\ref{Rem6}, the map $\mathcal H\mapsto{\bf q}^{\mathcal H}$ is a bijection between the set $\VHC(\pi)$ of valid hook configurations of a permutation $\pi$ and the set $\mathcal V(\pi)$ of valid compositions of $\pi$. Let us say a valid hook configuration $\mathcal H\in\VHC(\pi)$ is an \dfn{extremal hook configuration} if ${\bf q}^{\mathcal H}$ is a vertex of $\Fer_\pi$. 

\begin{question}
What can be said about extremal hook configurations? 
\end{question} 

Note that the discussion of $f$-vectors of binary nestohedra in Section~\ref{Subsec:BinaryNestohedra} allows us to compute the number of extremal hook configurations of a sorted permutation $\pi$ from the associated tree $\widehat T^\pi$ and binary building set $\mathcal B^\pi$ (see Theorem~\ref{Thm:Fertilitopes3}). Indeed, this number is $f_{\Nest({\bf y}^\pi)}(0)=\prod_{v}\deg(v)$, where the product ranges over all internal vertices of the tree $(\widehat T^\pi)^{\text c}$ (as defined in Section~\ref{Subsec:BinaryNestohedra}). 

\section{Acknowledgments}
The author thanks Akiyoshi Tsuchiya for a helpful conversation about integral polytopes. He also thanks Vincent Pilaud for pointing out the facts about $f$-vectors and $h$-vectors discussed in Section~\ref{Subsec:BinaryNestohedra} and for giving several other very helpful comments about the manuscript. The author was supported by a Fannie and John Hertz Foundation Fellowship and an NSF Graduate Research Fellowship.

\end{document}